\documentclass{article}

\usepackage{amsmath}
\usepackage{amssymb}
\usepackage{amsthm}
\usepackage{nccmath}					% \useshortskip before align
\usepackage{mathabx}
\usepackage{mathrsfs}
\usepackage{stmaryrd}
\usepackage{wasysym}
\usepackage{bbm}		% blackboard 1
\usepackage{accents}
\usepackage{scalerel}
\usepackage{array}
\newcolumntype{C}[1]{>{\hfil$}p{#1}<{$\hfil}}      % Centered array column with given width
\usepackage{mathtools}
\mathtoolsset{showonlyrefs=true}		% Show equation numbers only if referred
\usepackage[inline]{enumitem}
\usepackage{ifthen}
\usepackage{braket}						% defines \middle similar to \left and \right
\usepackage[normalem]{ulem}			% uline, uuline, uwave, xout, ...
\usepackage{xfrac}						% nice small slanted fractions
\usepackage{xcolor} 
\usepackage[a4paper, left=2.7cm, right=2.7cm, top=3.5cm, bottom=2cm]{geometry}
\usepackage[pdftex=true,hyperindex=true,pdfborder={0 0 0},colorlinks=true,allcolors=blue]{hyperref}
\usepackage{doi}
\usepackage[numbers]{natbib}
\newcommand{\ZZ}{\mathbb{Z}}			% The set of integers
\newcommand{\NN}{\mathbb{N}}			% The set of natural numbers
\newcommand{\RR}{\mathbb{R}}			% The set of reals
\newcommand{\GG}{\mathbb{G}}			% A group

\newcommand{\isdef}{\coloneqq}						% "Is Defined" symbol
\DeclarePairedDelimiter\abs{\lvert}{\rvert}			% Absolute value (using mathtools)
\DeclarePairedDelimiter\norm{\lVert}{\rVert}		% Norm (using mathtools)
\DeclarePairedDelimiter\floor{\lfloor}{\rfloor}	% Floor (using mathtools)
\DeclarePairedDelimiter\zinterval{\llbracket}{\rrbracket}		% Integer interval (using mathtools)

\newcommand{\dd}{\mathrm{d}}			% Differential
\newcommand{\ee}{\mathrm{e}}			% Euler's constant
\newcommand{\PP}{\mathbb{P}}
\newcommand\given[1][]{\mathop{#1\vert}}						% Conditional probability
\newcommand\relto[1][]{\,\mathop{#1\|}\,}						% Separator in D(...) (KL-divergence)
\newcommand\markovto[1][]{\xrightarrow{#1}}						% Markov transition (with prescribed kernel)
\newcommand{\indicator}[1]{\mathbbm{1}_{#1}}					% Indicator function
\newcommand{\TV}{\mathsf{TV}}										% Total variation
\newcommand{\mix}{\mathsf{mix}}									% Mixing time
\newcommand{\smallest}{\mathsf{min}}								% Smallest element
\newcommand{\distributionname}[1]{%					% Name of a distribution
    \ensuremath{\mathsf{#1}}%
}
\newcommand{\dPois}{\distributionname{Poisson}}		% Poisson distribution

\newcommand{\symb}[1]{\mathtt{#1}}								% A symbol

\newcommand{\vect}[1]{\underline{#1}}									% A vector (or finite sequence)

\newcommand{\pspace}[1]{\mathcal{#1}}	% Probability space, configuration space, subshifts, ...
\newcommand{\family}[1]{\mathscr{#1}}	% Family of of measures, ...
\newcommand{\field}[1]{\mathfrak{#1}}	% sigma-fields, ...

\newtheorem{theorem}{Theorem} %[section]
\newtheorem{corollary}{Corollary}
\newtheorem{question}{Question}

\newtheorem{lemma}{Lemma}[section]
\newtheorem{proposition}[lemma]{Proposition}
\newtheorem{observation}[lemma]{Observation}
\theoremstyle{definition}
\newtheorem{definition}[lemma]{Definition}
\newtheorem{remark}[lemma]{Remark}
\newtheorem{example}[lemma]{Example}
\AtBeginEnvironment{example}{%
  \pushQED{\qed}%
}
\AtEndEnvironment{example}{\popQED\endexample}
\AtBeginEnvironment{remark}{%
  \pushQED{\qed}%
}
\AtEndEnvironment{remark}{\popQED\endremark}
\AtBeginEnvironment{definition}{%
  \pushQED{\qed}%
}
\AtEndEnvironment{definition}{\popQED\enddefinition}

\usepackage{acro}
\DeclareAcronym{iid}{%
    short=i.i.d\acdot,
    long=independent and identically distributed,
    first-style=short
}
\DeclareAcronym{ie}{%
    short=i.e\acdot,
    long=that is,
    first-style=short
}
\DeclareAcronym{eg}{%
    short=e.g\acdot,
    long=for example,
    first-style=short
}
\DeclareAcronym{aka}{%
    short=a.k.a\acdot,
    long=also known as,
    first-style=short
}

\allowdisplaybreaks

\title{%
	Positive-rate PCA and IPS with stationary Bernoulli measures are rapidly forgetful
	\footnotetext{Last update:~\today}
}

\author{%
	Irène Marcovici
	\and
	Siamak Taati
}

\newcommand{\Addresses}{{
    \footnotesize
    \begin{samepage}
    	\noindent I.~Marcovici\\
    	\textsc{Université de Rouen Normandie, LMRS, 76801 Saint-Étienne-du-Rouvray, France.}\\
    	\indent\emph{E-mail address}: \texttt{\href{mailto:irene.marcovici@univ-rouen.fr}{irene.marcovici@univ-rouen.fr}}
    \end{samepage}
		
	\medskip
	
	\begin{samepage}
		\noindent S.~Taati\\
		\textsc{Department of Mathematics, American University of Beirut, Beirut, Lebanon.}\\
		\textsc{Center for Advanced Mathematical Sciences, American University of Beirut, Beirut, Lebanon.}\\
		\indent\emph{E-mail address}: \texttt{\href{mailto:siamak.taati@aub.edu.lb}{siamak.taati@aub.edu.lb}}
	\end{samepage}
	
	\medskip
}}

\date{}

\begin{document}

\maketitle

\begin{abstract}
	We prove that every probabilistic cellular automaton with strictly positive transition probabilities that admits a stationary Bernoulli measure is exponentially ergodic.  Moreover, the mixing time of any finite region in such a system is logarithmic in the diameter of the region.
	A similar result holds in continuous time for positive-rate, finite-range interacting particle systems.
    The proofs use entropy, and rely on a representation of the system as a perturbation of another system with noise.  The ergodic behaviour results from a competition between the accumulation of randomness due to noise and the diffusion of randomness due to local information exchange.
    %The proofs are based on entropy contraction and a bootstrap argument.
    %Entropy contraction is established through representing the system as a perturbation of another system with noise.
	%While the interacting particle systems that have stationary Bernoulli measures are fairly limited, we will illustrate that the probabilistic cellular automata that admit stationary Bernoulli measures form a rich class of models.
    We show that, in two and higher dimensions, the positive-rate probabilistic cellular automata that admit stationary Bernoulli measures are algorithmically indistinguishable from those that do not.
	
	\medskip
		
	
	\smallskip
	
	    \textit{Secondary:}
		%60J10, %Markov chains (discrete-time Markov processes on discrete state spaces)
		%60J27, %Continuous-time Markov processes on discrete state spaces 
		%37A50, %Dynamical systems and their relations with probability theory and stochastic processes
		82C20, %Dynamic lattice systems (kinetic Ising, etc.) and systems on graphs in time-dependent statistical mechanics
		% 37B15 %Dynamical aspects of cellular automata
        68Q80. %Cellular automata (computational aspects) {For cellular automata as dynamical systems, see 37B15}

	%Temporary ********************************
	\vspace{-3em}
	\renewcommand{\contentsname}{}
	{\footnotesize\tableofcontents}
	%******************************************
\end{abstract}

\section{Introduction}
\label{sec:intro}

A \emph{probabilistic cellular automaton} (PCA) is a discrete-time Markov process on configurations of symbols on a lattice in which the symbols are updated synchronously, at random, with probabilities prescribed by a local transition rule.  More specifically, the \emph{configurations} of the model are elements of~$\Sigma^{\ZZ^d}$, where $\Sigma$ is a finite alphabet and $d\in\{1,2,\ldots\}$ is the dimension of the lattice.  The \emph{local transition rule} is a stochastic matrix $\varphi\colon\Sigma^N\times\Sigma\to[0,1]$, where $N\subseteq\ZZ^d$ is a finite set called the dependence \emph{neighbourhood} of~$\varphi$.  From a configuration~$x\in\Sigma^{\ZZ^d}$, the system transitions into a new configuration in which the symbol at each site $k$ is drawn, at random, according to distribution~$\varphi\big((x_{k+i})_{i\in N},\cdot\big)$, independently of the other sites and of all the past transitions.

We say that a PCA is \emph{ergodic} if it admits a unique stationary measure $\pi$ and the distribution of the system converges weakly to~$\pi$, irrespective of its initial distribution.  Denoting the global transition kernel of the PCA by $\Phi$, stationarity of $\pi$ is described as $\pi\Phi=\pi$, and ergodicity means that, for every initial measure $\mu$, $\mu\Phi^n\to\pi$ as $n\to\infty$, in that the marginal $(\mu\Phi^n)_J$ of $\mu\Phi^n$ on any finite region~$J\subseteq\ZZ^d$ converges to the corresponding marginal of~$\pi$.  A \emph{Bernoulli measure} on~$\Sigma^{\ZZ^d}$ refers to a product measure $\lambda=\bigotimes_{k\in\ZZ^d}p$ with the same marginal~$p$ at every site of the lattice.  The total variation distance between two probability distributions $p$ and $q$ on a finite set is denoted by~$\norm{p-q}_\TV$.

Our first result states that, every PCA that has strictly positive transition probabilities and a stationary Bernoulli measure is ergodic with an exponentially fast convergence.
\begin{theorem}[Exponential ergodicity of positive-rate PCA with stationary Bernoulli measures]
\label{thm:main:PCA:ergodicity}
	Let $\Phi$ be a $d$-dimensional PCA with a strictly positive local transition rule $\varphi\colon\Sigma^N\times\Sigma\to(0,1)$, and suppose that $\Phi$ admits a stationary Bernoulli measure $\lambda$.
	Then, there exist constants $\alpha,\beta>0$ %,a,b>0$
	such that
	\begin{equation}
		\norm[\big]{(\mu\Phi^t)_J - \lambda_J}_\TV \leq \alpha\ee^{-\beta t} n^{d/2}
	\end{equation}
	for every initial measure $\mu$, every finite region $J\subseteq\ZZ^d$ with diameter~$n$, %and every $t\geq a\log n + b$.
	and every $t\geq 0$.
\end{theorem}

% A simple proof of the above theorem is given in~Section~\ref{sec:PCA:ergodicity}.  The constants $\alpha$ and $\beta$ can be explicitly computed from~$\varphi$ and~$q$ (see the proof).

Our second result is a continuous-time variant of the above theorem for (finite-range) interacting particle systems.  In an \emph{interacting particle system} (IPS), the updates at different sites occur asynchronously, triggered by independent Poisson clocks.  Namely, let $\varphi\colon\Sigma^N\times\Sigma\to[0,1]$ be a local transition rule and $(\xi_k)_{k\in\ZZ^d}$ be a family of independent Poisson processes (\emph{clocks}) with the same rate~$c>0$, one attached to each site.  (We can always choose the unit of time such that $c=1$.)  Every time the clock $\xi_k$ ticks, the symbol at site~$k$ is updated, at random, according to the distribution prescribed by~$\varphi$.  The notions of invariance and ergodicity are described analogously in terms of the transition semigroup~$(\Phi^t)_{t\geq 0}$ of the IPS.

\begin{theorem}[Exponential ergodicity of positive-rate IPS with stationary Bernoulli measures]
\label{thm:main:IPS:ergodicity}
	Let $\Phi$ be a $d$-dimensional IPS with clock rate~$1$ and a strictly positive local transition rule $\varphi\colon\Sigma^N\times\Sigma\to(0,1)$, and suppose that $\Phi$ admits a stationary Bernoulli measure $\lambda$.
	Then, there exist constants $\alpha,\beta>0$ %,a,b>0$
	such that
	\begin{equation}
		\norm[\big]{(\mu\Phi^t)_J - \lambda_J}_\TV \leq \alpha\ee^{-\beta t} n^{d/2}
	\end{equation}
	for every initial measure $\mu$, every finite region $J\subseteq\ZZ^d$ with diameter~$n$, %and every $t\geq a\log n + b$.
	and every $t\geq 0$.
\end{theorem}

Examples of models to which Theorems~\ref{thm:main:PCA:ergodicity} and~\ref{thm:main:IPS:ergodicity} apply are presented in Section~\ref{sec:characterizations} below.
An immediate consequence of these theorems is that the mixing time of any finite region in a positive-rate PCA or IPS admitting a stationary Bernoulli measure is logarithmic in the diameter of the region.
The \emph{mixing time} of a finite region $J\subseteq\ZZ^d$ with \emph{error margin} $\varepsilon>0$ is
\begin{equation}
	t_\mix(J,\varepsilon) \isdef \inf\big\{t\geq 0: \text{$\norm[\big]{(\mu\Phi^s)_J-\lambda_J}_\TV<\varepsilon$ for all $\mu$ and $s\geq t$}\big\} \;.
\end{equation}

\begin{corollary}[Mixing times of positive-rate PCA/IPS with stationary Bernoulli measures]
\label{cor:main:mixing-time}
	Let $\Phi$ be a $d$-dimensional PCA or IPS with a strictly positive local transition rule $\varphi\colon\Sigma^N\times\Sigma\to(0,1)$, and suppose that $\Phi$ admits a stationary Bernoulli measure $\lambda$.
	Then,
%	there exist constants $a',b'>0$ such that $t_\mix(J,\varepsilon)\leq a'\log n + b'$ for every finite region $J\subseteq\ZZ^d$ with diameter~$n$.
%	More specifically,
	\begin{equation}
		t_\mix(J,\varepsilon) \leq
%			\max\left\{
%			\frac{d-1}{2\beta}\log n + \frac{\log\alpha/\varepsilon}{\log\beta}
%			,
%			a\log n + b
%			\right\}\;.
			\frac{d}{2\beta}\log n + \frac{\log\alpha - \log\varepsilon}{\beta}
	\end{equation}
    for every $J\subseteq\ZZ^d$ with diameter~$n$,
	where $\alpha$ and $\beta$ are the constants in Theorems~\ref{thm:main:PCA:ergodicity} or~\ref{thm:main:IPS:ergodicity}, respectively.
\end{corollary}

The motivation to study positive-rate PCA and IPS with stationary Bernoulli measures is twofold.
The first reason is the apparent trade-off between heat dissipation~\cite{Landauer1961,Bennett1982} and noise-resilience~\cite{Neumann1956,GR1988,Gacs1986,Gacs2001,MST2019,CG2021} in physical realizations of computation.  Reversible (deterministic) CA have long been studied as mathematical models of computation in which all operations are performed reversibly so as to prevent heat dissipation~\cite{Toffoli1977,MH1989,TM1990,Morita2008,Kari2018}.  A perturbation of a reversible CA with additive noise gives rise to a PCA that admits the uniform Bernoulli measure as a stationary measure (see Example~\ref{exp:surjective-CA+noise}).  Theorem~\ref{thm:main:PCA:ergodicity} and Corollary~\ref{cor:main:mixing-time} generalize an earlier result~\cite{Taati2021} (and its precursor~\cite[Theorem~4.1]{MST2019}) stating that all such perturbations are ergodic with rapid convergence.  In particular, every finite region in such a noisy computation model forgets its initial data in a logarithmic number of steps, making the system incapable of performing anything but trivial computations at large scales.

The second reason for interest in PCA and IPS with stationary Bernoulli measures arises in the context of statistical mechanics, as a testing ground to study the more general family of PCA and IPS that admit stationary Gibbs measures. This more general family includes reversible Gibbs samplers (\ac{aka} Glauber dynamics, or stochastic Ising models)~\cite{Liggett1985,Martinelli1999} as well as many non-reversible models.  In this context, the question of convergence towards equilibrium remains open:
\begin{question}
\label{q:gibbs-stationary}
    Suppose that a positive-rate PCA or IPS admits a stationary Gibbs measure.  Does the distribution of the system starting from any initial condition necessarily converge weakly to the set of Gibbs measures with the same specification?
\end{question}
Building upon the earlier partial result of Holley and Stroock~\cite{HS1977}, Jahnel and Köppl have recently settled this question in the affirmative in the special case of reversible IPS in one and two dimensions~\cite{JK2025}. Theorems~\ref{thm:main:PCA:ergodicity} and~\ref{thm:main:IPS:ergodicity} above establish another special case, namely when the stationary Gibbs measure is in fact a Bernoulli measure, but without the reversibility assumption or any restriction on the dimension.  Earlier partial results related to Question~\ref{q:gibbs-stationary} are all limited to shift-invariant measures.
Holley and others have shown that, if a positive-rate IPS or PCA admits a reversible Gibbs measure, then all its shift-invariant stationary measures are Gibbs for the same specification, and moreover, starting from any shift-invariant measure, the system converges to the set of (shift-invariant) Gibbs measures with the same specification~\cite{Holley1971, HS1975,Sullivan1976,MP1977,Vasilyev1978,KV1980}.
The first statement holds, even without reversibility, for positive-rate IPS and PCA that admit stationary Gibbs measures~\cite{Kuensch1984,DLR2002,JK2019}.

The proofs of all the above-mentioned results rely on relative entropy (or, equivalently, free energy).  
For the dynamics of the system restricted to shift-invariant measures, the relative entropy per site serves as a (semi-continuous) Lyapunov function.
The results of Holley and Stroock~\cite{HS1977} and Jahnel and Köppl~\cite{JK2025} involve a more intricate analysis of the time derivative of the relative entropy of finite regions.
Our proof also relies on analyzing the evolution of the relative entropy of finite regions, but it has a more information-theoretic flavour as we now briefly sketch.

The starting point %in the proof of Theorems~\ref{thm:main:PCA:ergodicity} and~\ref{thm:main:IPS:ergodicity}
is to represent the model as a perturbation of another model with random (zero-range, memoryless) noise, in such a way that both components preserve the same Bernoulli measure~$\lambda$ (Sections~\ref{sec:PCA:noise-decomposition} and~\ref{sec:IPS:noise-decomposition}).  The two components have competing effects on the relative entropy of a finite region~$J\subseteq\ZZ^d$.
The noise acts as a contraction, leading to exponential decay of the relative entropy, whereas the other component has a diffusive effect, allowing relative entropy to leak into~$J$ through its boundary (Sections~\ref{sec:PCA:entropy-evolution} and~\ref{sec:IPS:entropy-almost-depletion}).
In discrete time, a simple argument shows that the noise prevails, yielding exponential convergence towards~$\lambda$ (Section~\ref{sec:PCA:entropy-evolution}).
The same holds in continuous time, but the proof requires a more elaborate bootstrapping argument (Section~\ref{sec:bootstrap:continuous-time}),  which, in a simpler setting, was introduced in an earlier proof of a special case of Theorem~\ref{thm:main:PCA:ergodicity}~\cite{Taati2021}.
For IPS, we exploit the interpretation in terms of asynchronous updating (Observation~\ref{obs:IPS:generator-in-terms-of-asynchronous} and~\ref{prop:IPS:entropy-change:in-terms-of-asynchronous}), which contrasts with the synchronous updating scheme in PCA.  We also require a concentration inequality controlling the speed of information propagation when updates are driven by Poisson clocks (Section~\ref{sec:IPS:influence-propagation}).

The use of entropy in the study of Markov processes is rooted in the original work of Boltzmann, and was first formulated in the context of Markov chains by Rényi~\cite{Renyi1961}.  In the setting of IPS and PCA, relative entropy has also been used to establish exponential convergence via logarithmic Sobolev inequalities (see \ac{eg},~\cite{GZ2003,Caputo2022}).  Another application in this context is Dawson's information-theoretic approach to problems of uniqueness and ergodicity~\cite{Dawson1973,Dawson1974,Dawson1975}.

% \comment{To write:}
% The use of entropy in the study of Markov processes is rooted in the original work of Boltzmann, and was formulated in rigorous form by Rényi~\cite{Renyi1961}.  The applications of en the context of PCA and IPS, in addition to the above, include the work of Dawson who derived information-theoretic sufficient criteria for the uniqueness of stationary measures and ergodicity~\cite{}, ... hydrodynamic limit, and the entropy proof by Ramirez and Varadhan~\cite{RV1996} of an earlier result of Mountford~\cite{Mountford1995} stating that positive-rate finite-range IPS in one dimension cannot have periodic behaviour.  Relation to log-Sobolev inequalities?

Theorems~\ref{thm:main:PCA:ergodicity} and~\ref{thm:main:IPS:ergodicity} can be generalized in several directions, but we present them in the simplest non-trivial setting to emphasize the key ideas.  First, the results extend to non-homogeneous PCA and IPS, in which the local transition rules may vary across sites and the stationary product measure need not be shift-invariant, provided that the local rules have uniformly bounded ranges and are uniformly strictly positive.  In this setting, the underlying lattice can be replaced by any locally finite graph of sub-exponential growth, at least in the case of PCA (Remark~\ref{rem:graph-with-sub-exponential-growth}).  We expect that the finite-range assumption on the local rules can also be relaxed, as long as the dependence on distant sites remains sufficiently weak.

Despite the strong analogy between Theorems~\ref{thm:main:PCA:ergodicity} and~\ref{thm:main:IPS:ergodicity} and the similarity of their proofs, the corresponding models appear to differ substantially in nature, due to the contrast between synchronous and asynchronous updating.  We show that, in two and higher dimensions, whether a given positive-rate PCA admits a stationary Bernoulli measure is algorithmically undecidable (Theorem~\ref{thm:PCA:undecidability}).  This suggests that, despite their rapid ergodicity, such models still exhibit some form of complexity, albeit in their one-step transition kernel rather than their asymptotic behaviour.  By contrast, we conjecture that, in any dimension, positive-rate IPS that admit stationary Bernoulli measures have a finitary characterization (see Section~\ref{sec:characterizations:IPS}).
In one dimension, such finitary characterizations are known to exist for both classes of models~\cite[Chapter~16]{TVS+1990}, \cite{FM2020}.

% \comment{To write:}
% Despite the strong analogy between the statements of Theorems~\ref{thm:main:PCA:ergodicity} and~\ref{thm:main:IPS:ergodicity} and the similarity of their proofs, the models they talk about are quite different in nature. ... Positive-rate IPS with stationary Bernoulli measures have a local characterization (?).  In contrast, the positive-rate PCA that admit stationary Bernoulli measures form a much richer family of systems, to the point that they cannot be algorithmically distinguished from other PCA.  See Section~\ref{...}.

The remainder of the paper is organized as follows.
Section~\ref{sec:preliminaries} introduces the setting and reviews the necessary preliminaries, in particular on IPS and relative entropy.  In Section~\ref{sec:non-interacting-MCs}, we present (fairly standard) entropy contraction inequalities (\ac{aka}, strong data processing inequalities) for finite families of non-interacting Markov chains with synchronous or asynchronous updating.  These inequalities are used in the subsequent sections to quantify the effect of noise.  Theorems~\ref{thm:main:PCA:ergodicity} and~\ref{thm:main:IPS:ergodicity} are proven in Sections~\ref{sec:PCA:ergodicity} and~\ref{sec:IPS:ergodicity}, respectively.  In Section~\ref{sec:characterizations}, we present some examples and address the problem of identifying PCA and IPS that admit stationary Bernoulli measures.

% \comment{%
% \begin{itemize}
%     \item General references: \cite{TVS+1990} and \cite{MM2014a} for PCA; \cite{Liggett1985} and ?? for IPS
%     % \item Interest in computer science : What to say?  What to cite?
%     % \item Interest in statistical physics: What to say?  What to cite?
%     % \item For PCA: invariance of Bernoulli measure appear and the question of ergodicity appear in the context of reversible computation with noisy components~\cite{MST2019}.  Theorem~\ref{thm:main:PCA:ergodicity} and Corollary~\cite{cor:main:mixing-time} generalize the earlier result about surjective CA + noise~\cite{Taati2021}.
%     % \item For PCA/IPS: The problem we consider is a special case of the more difficult problem of whether a PCA/IPS that has an stationary Gibbs measure converges to the set of Gibbs measures with the same specification.  The usual approach is via entropy.  Many references (Glauber/reversible/non-reversible dynamics?), Holley, ..., Jahnel and Köppl seem to prove convergence to the set of Gibbs measures in one and two dimensions if the IPS has a reversible Gibbs measure.
% \end{itemize}
% }

\paragraph{Acknowledgments.}
We thank Régine Marchand, Jérôme Casse, and Pierre Youssef for helpful discussions.

% \comment{Add}

% \comment{Régine Marchand? Jérôme Casse? Pierre Youssef?}

% \comment{Funding acknowledgment}

\section{Preliminaries}
\label{sec:preliminaries}

We write $A\Subset B$ to indicate $A$ is a \emph{finite} subset of~$B$.
We use $\zinterval{i,j}$ for the integer interval $[i,j]\cap\ZZ$. 

\subsection{Configurations, measures, transition kernels}
Throughout this article, $\Sigma$ stands for a finite alphabet.
The configuration space $\Sigma^{\ZZ^d}$ is given the product topology, which is compact and metrizable, and the Borel $\sigma$-algebra. %, which we denote by~$\field{F}$.

A partial configuration $w\colon J\to \Sigma$ with finite $J\Subset\ZZ^d$ is called a \emph{pattern}.  The \emph{cylinder} with base $w$ is the set
\begin{equation}
    [w] \isdef \{x\in\Sigma^{\ZZ^d}: x_J=w\}
\end{equation}
of all configurations that agree with $w$ over its domain.
The cylinders are clopen and form a basis for the product topology on~$\Sigma^{\ZZ^d}$.  Together with the empty set, they also form a semi-algebra generating the Borel $\sigma$-algebra on~$\Sigma^{\ZZ^d}$.

The translation (or \emph{shift}) of a configuration $x\in\Sigma^{\ZZ^d}$ by a vector $k\in\ZZ^d$ is denoted by $\sigma^k x$, so that $(\sigma^k x)_i\isdef x_{k+i}$ for $i\in\ZZ^d$.
%The translation (or \emph{shift}) action $\sigma$ is continuous.

The Banach space of all continuous functions $f\colon\Sigma^{\ZZ^d}\to\RR$ with the uniform norm is denoted by~$C(\Sigma^{\ZZ^d})$.  A function $f\colon\Sigma^{\ZZ^d}\to\RR$ is said to be \emph{local} if there is a finite set $J\Subset\ZZ^d$, called the \emph{base} of~$f$, such that $f(x)$ is uniquely determined by the restriction~$x_J$.  In other words, $f$ is local if it is a linear combination of indicators of cylinders.
Local functions form a countable dense linear subspace of~$C(\Sigma^{\ZZ^d})$.

The space of all probability measures on $\Sigma^{\ZZ^d}$ is denoted by~$\family{P}(\Sigma^{\ZZ^d})$.  With the weak topology, this is again compact and metrizable.  The marginal of a measure $\mu\in\family{P}(\Sigma^{\ZZ^d})$ on $J\subseteq\ZZ^d$ is denoted by $\mu_J$.  A sequence $\mu_1,\mu_2,\ldots\in\family{P}(\Sigma^{\ZZ^d})$ converges weakly to a measure $\mu\in\family{P}(\Sigma^{\ZZ^d})$ if and only if $\mu_n([w])\to\mu([w])$ for every cylinder $[w]$, which is the case, if and only if, the marginals of $\mu_n$ on every finite set $J\Subset\ZZ^d$ converge to the corresponding marginal of~$\mu$.

% Let $\Phi\colon\Sigma^{\ZZ^d}\times\field{F}\to[0,1]$ be a probability transition kernel on $\Sigma^{\ZZ^d}$.  Following the usual convention, we use the left multiplication notation $f\mapsto \Phi f$ for the operator induced by $\Phi$ on bounded measurable functions, and the right multiplication notation $\mu\mapsto\mu\Phi$ for the dual operator on probability measures, so that
% \begin{alignat}{2}
%     (\Phi f)(x) &\isdef\int f(y)\dd\Phi(x,\dd y)    &&\quad \text{for $x\in\Sigma^{\ZZ^d}$,}\\
%     (\mu\Phi)(E) &\isdef\int\Phi(x,E)\mu(\dd x)     &&\quad \text{for $E\in\field{F}$.}
% \end{alignat}
% We say that $\Phi$ has the \emph{Feller property} if $\Phi f\in C(\Sigma^{\ZZ^d})$ for every $f\in C(\Sigma^{\ZZ^d})$, or equivalently, if $\mu\mapsto\mu\Phi$ is continuous on~$\family{P}(\Sigma^{\ZZ^d})$.
% Likewise, a semigroup $(\Phi^t)_{t\geq 0}$ of transition kernels is said to have the \emph{Feller property} if the action $(t,\mu)\mapsto\mu\Phi^t$ is continuous.

Let $\Phi\colon\pspace{A}\times\field{F}_{\pspace{B}}\to[0,1]$ be a probability transition kernel from a measurable space $(\pspace{A},\field{F}_{\pspace{A}})$ to a measurable space $(\pspace{B},\field{F}_{\pspace{B}})$.  Following the usual convention, we use the left multiplication notation $f\mapsto \Phi f$ for the operator induced by $\Phi$ on bounded measurable functions on~$\pspace{Y}$, and the right multiplication notation $\mu\mapsto\mu\Phi$ for the dual operator on probability measures on~$\pspace{X}$, so that
\begin{alignat}{2}
    (\Phi f)(x) &\isdef\int f(y)\dd\Phi(x,\dd y)\;,    &&\quad \text{for $x\in\pspace{A}$,}\\
    (\mu\Phi)(E) &\isdef\int\Phi(x,E)\mu(\dd x)\;,     &&\quad \text{for $E\in\field{F}_{\pspace{B}}$.}
\end{alignat}
If $X$ and $Y$ are random variables with values in $\pspace{A}$ and $\pspace{B}$, respectively, we write $X\markovto[\Phi]Y$ to indicate that, conditioned on~$X$, $Y$ is distributed according to~$\Phi(X,\cdot)$.
Following the usual convention (\ac{eg}~\cite{CT2006}), we write $X\markovto Y\markovto Z$ if random variables $X, Y, Z$ form a Markov chain, that is, if given $Y$, $X$ and $Z$ are independent.  We write $X\sim \mu$ to indicate $X$ is a random variable with distribution~$\mu$.

% \begin{itemize}
% 	\item Notation $A\Subset B$ (?)
% 	\item Notation $\zinterval{i,j}$ for integer intervals
% 	\item The space $\Sigma^{\ZZ^d}$
% 	\item Shift action $\sigma$
% 	\item Topology of $\Sigma^{\ZZ^d}$; cylinder notation $[w]$
% 	\item $\sigma$-algebra $\field{F}$ of $\Sigma^{\ZZ^d}$
% 	\item Space of probability measures $\family{P}(\Sigma^{\ZZ^d})$
% 	\item $C(\Sigma^{\ZZ^d})$ and local functions; base of a local function
% 	\item Projection $\mu_J$ of a measure $\mu\in\family{P}(\Sigma^{\ZZ^d})$ on a set $J\subseteq\ZZ^d$.
% 	\item Weak topology on $\family{P}(\Sigma^{\ZZ^d})$
% 	\item Introduce the $\sigma$-algebra $\field{F}_J$ for $J\subseteq\ZZ^d$? or the projection $\mu\mapsto\mu_J$?
% 	\item Transition kernels?
% \end{itemize}

\subsection{PCA and influence region}
The global transition kernel of a PCA with local transition rule $\varphi\colon\Sigma^N\times\Sigma\to[0,1]$ is given by
\begin{equation}
    \Phi(x,[w]) \isdef \prod_{k\in J} \varphi\big((\sigma^k x)_N, w_k\big)
\end{equation}
for a configuration $x\in\Sigma^{\ZZ^d}$ and a cylinder $[w]$, where $w\in\Sigma^J$.
Note that if $f\colon\Sigma^{\ZZ^d}\to\RR$ is local, so is $\Phi f$. %hence $\Phi$ has the Feller property.
In analogy with the continuous time setting, we say that $\Phi$ has \emph{positive rates} if $\varphi$ is strictly positive.

Given a set $A\subseteq\ZZ^d$, we let $N(A)\isdef\{a+i: a\in A, i\in N\}=A+N$ be the set of all neighbours of the elements of~$A$.  For $t\in\NN$, the set $N^t(A)$ contains all sites whose states at time $0$ may influence the state of $A$ at time~$t$, in the sense that, for a random trajectory $(X^t)_{t\in\NN}$, we have
\begin{equation}
    \PP(X^t_A\in\cdot\given X^0)=\PP\big(X^t_A\in\cdot\given[\big] X^0_{N^t(A)}\big)
\end{equation}
almost surely.

We refer to~\cite{TVS+1990,MM2014a} for further details about the setting.

\subsection{IPS and asynchronous updating}
\label{sec:IPS:asynchronous-updating}

An IPS with local transition rule $\varphi\colon\Sigma^N\times\Sigma\to[0,1]$ and Poisson clocks $(\xi_k)_{k\in\ZZ^d}$ evolves as follows.  Let $(X^t)_{t\geq 0}$ denote the random trajectory of the system.
A tick of the clock $\xi_k$ at time $t$ triggers an update at site~$k$, so that the symbol $X^t_k$ is resampled according to the distribution $\varphi\big((\sigma^k X^{t-})_N,\cdot)$.  The updates are independent of one another and of the Poisson clocks.
Without loss of generality, the Poisson clocks are assumed to have rate~$1$ throughout this article.
We let $\Phi^t$ denote the transition kernel from time~$0$ to time~$t$.

Alternatively, $\Phi=(\Phi^t)_{t\geq 0}$ is the Markov semigroup corresponding to a generator $L$, which is defined by
\begin{equation}
	(Lf)(x) \isdef
		\sum_{k\in\ZZ^d} \sum_{b\in\Sigma}\varphi\big((\sigma^k x)_N, b\big)\big(f(\zeta_{k\to b}x)-f(x)\big)
\end{equation}
for every local function $f\colon\Sigma^{\ZZ^d}\to\RR$, where %$x^{(k\to b)}$
$\zeta_{k\to b}x$
stands for the configuration that has symbol~$b$ at position~$k$ and agrees with $x$ everywhere else.
Note that $Lf$ is again a local function.

For further details on the setting and technical aspects (including the questions of existence and uniqueness), we refer to~\cite{Liggett1985,Swart2025}.

The connection between infinite-volume IPS and finite-volume models with asynchronous updating scheme can be described as follows.
Given $k\in\ZZ^d$, define $\widehat{\Phi}_k$ to be the transition kernel on $\Sigma^{\ZZ^d}$ that represents the updating of the symbol at site~$k$ using~$\varphi$.  More precisely,
\begin{equation}
	%\widehat{\Phi}_k(x,f) \isdef
    \big(\widehat{\Phi}_k f\big)(x) \isdef
		\sum_{b\in\Sigma}\varphi\big((\sigma^k x)_N, b\big)f(\zeta_{k\to b}x)
	\label{eq:asynchronous-kernel:single-site}
\end{equation}
for every $x\in\Sigma^{\ZZ^d}$ and $f\in C(\Sigma^{\ZZ^d})$.
Given $J\Subset\ZZ^d$, let
\begin{equation}
	\widehat{\Phi}_J \isdef \frac{1}{\abs{J}}\sum_{k\in J}\widehat{\Phi}_k \;.
	\label{eq:asynchronous-kernel:finite-set}
\end{equation}
In words, $\widehat{\Phi}_J$ is the transition kernel for the process of picking a site $k$ uniformly at random from~$J$ and updating the symbol at~$k$ using~$\varphi$.
We call the kernels $\widehat{\Phi}_J$ (for $J\Subset\ZZ^d$) the \emph{asynchronous updating} kernels associated with the local transition rule~$\varphi$.

\begin{observation}[Generator in terms of asynchronous updating]
\label{obs:IPS:generator-in-terms-of-asynchronous}
	Let $L$ be the generator for an IPS with local transition rule $\varphi\colon\Sigma^N\times\Sigma\to[0,1]$ and clock rate~$1$.  Then,
	\begin{equation}
		(Lf)(x) = \abs{J}\Big(\big(\widehat{\Phi}_J f\big)(x)-f(x)\Big)
	\end{equation}
	for every $J\Subset\ZZ^d$ and every local function $f$ with base~$J$.
	% \comment{%
	% In particular, $L$ ``defines a signed kernel'' with
	% \begin{equation}
	% 	L(x,[w]) = \abs{J}\Big(\widehat{\Phi}_J(x,[w]) - \indicator{[w]}(x)\Big)
	% \end{equation}
	% for every $J\Subset\ZZ^d$ and $w\in\Sigma^J$.
	% }
\end{observation}

Thus, one can equivalently think of $L$ as a finitely additive signed kernel on the algebra generated by cylinders, where
\begin{equation}
    L(x,[w]) = \abs{J}\Big(\widehat{\Phi}_J(x,[w]) - \indicator{[w]}(x)\Big)
\end{equation}
for every $x\in\Sigma^{\ZZ^d}$, $J\Subset\ZZ^d$ and $w\in\Sigma^J$.
In particular, $L$ maps a probability measure $\mu\in\family{P}(\Sigma^{\ZZ^d})$ into a finitely additive, signed measure $\mu L$, where
\begin{equation}
    (\mu L)([w]) \isdef \abs{J}\Big(\big(\mu\widehat{\Phi}_J\big)([w]) - \mu([w])\Big)
\end{equation}
for every cylinder~$[w]$.

From Observation~\ref{obs:IPS:generator-in-terms-of-asynchronous}, it immediately follows that:
\begin{observation}[Sufficient condition for stationarity]
\label{obs:IPS:local-preservation}
	%If (for one, and hence all $k\in\ZZ^d$) $\widehat{\Phi}_k$ preserves a measure $\lambda$, then so does the IPS~$\Phi$.
    A probability measure $\lambda$ is stationary for an IPS $\Phi$ if $\lambda\widehat{\Phi}_k=\lambda$ for every $k\in\ZZ^d$.
\end{observation}
The converse is not true, as we shall see in Example~\ref{exp:IPS:stationary-but-not-locally}.

% \comment{%
% TO DO:
% \begin{itemize}
% 	\item Make sense of notation $\mu L$
% 	\item Question in the last section: Is the converse of the above observation true as well?
% \end{itemize}
% }
% \irene{Converse statement is false, cf. Example~5.}

% \subsection{Ergodicity and mixing times}

% Two aspects of the ``speed of convergence'':
% \begin{itemize}
% 	\item Exponential convergence.  Martinelli's example in which the convergence is sub-exponential.  A non-trivial example in which the convergence is super-exponential (XOR + noise)
% 	\item Mixing time of a finite region (time it forgets initial condition up to a given precision)
% 	\item The (trivial) proof of Corollary~\ref{cor:main:mixing-time}?
% \end{itemize}

\subsection{Relative entropy}

% \begin{itemize}
% 	\item Notation for relative entropy; chain rule in this notation
% 	\item Notation $A\markovto B\markovto C$
% 	\item Type of a random variable (?)
% 	\item Recall that if $A\markovto[\theta]B$ and $A'\markovto[\theta]B'$, then $D\big((B\given A)\relto[\big](B'\given A')\big)=0$.
% 	\item Notation $D_J(\mu\relto\lambda)$ for $D(\mu_J\relto\lambda_J)$.
% \end{itemize}

% Let $p,q\colon \Gamma\to[0,1]$ be probability distributions on a finite set $\Gamma$.  Recall that the \emph{relative entropy} (\ac{aka}, Kullback-Leibler divergence) between $p$ and $q$ is defined as
% \begin{equation}
%     D(p\relto q) \isdef 
%         \begin{dcases}
%             \sum_{a\in\Gamma} p(a) \log \frac{p(a)}{q(a)}       & \text{if $p\ll q$,} \\
%             \infty                                              & \text{otherwise,}
%         \end{dcases}
% \end{equation}
% with the conventions $0\log 0=0$ and $0\log\frac{0}{0}=0$.

% For random variables $A\sim p$ and $A'\sim q$, we write $D(A\relto A')$ for $D(p\relto q)$.

The relative entropy (\ac{aka}, Kullback-Leibler divergence) between two discrete probability distributions $p,q\colon\Sigma\to[0,1]$ is denoted by~$D(p\relto q)$ (see \ac{eg},~\cite{CT2006}).
For random variables $A\sim p$ and $A'\sim q$, we write $D(A\relto A')$ for $D(p\relto q)$, hence
\begin{equation}
    D(A\relto A') \isdef \sum_{a\in\Sigma} \PP(A=a)\log\frac{\PP(A=a)}{\PP(A'=a)} \;,
\end{equation}
with the usual conventions $0\log 0=0$, $0\log\frac{0}{0}=0$, and $x\log\frac{x}{0}=\infty$ for $x>0$ so as to make $D$ continuous.
Given pairs  of discrete random variables $(A,C)$ and $(A',C')$ with values from $\Sigma\times\Gamma$, we use the following notation for conditional relative entropy:
% \begin{align}
%     D\big((A\given C) \relto[\big] (A'\given C')\big) &\isdef
%         \sum_{c\in\Gamma}\sum_{a\in\Sigma}\PP(A=a,C=c) \log\frac{\PP(A=a\given C=c)}{\PP(A'=a\given C'=c)} \;, \\
%     D\big((A\given C) \relto[\big] A'\big) &\isdef
%         \sum_{c\in\Gamma}\sum_{a\in\Sigma}\PP(A=a,C=c) \log\frac{\PP(A=a\given C=c)}{\PP(A'=a)} \;,
% \end{align}
\begin{align}
    D\big((A\given C) \relto[\big] (A'\given C')\big) \isdef
        \sum_{c\in\Gamma}\sum_{a\in\Sigma}\PP(A=a,C=c) \log\frac{\PP(A=a\given C=c)}{\PP(A'=a\given C'=c)} \;.
\end{align}
With this notation, the chain rule of relative entropy can be expressed as follows:
\begin{equation}
    D\big((A,C)\relto (A',C')\big) =
        D(C\relto C') + D\big((A\given C) \relto[\big] (A'\given C')\big) \;.
\end{equation}

Note that, for $D\big((A\given C) \relto[\big] (A'\given C')\big)$ to make sense, $(A,C)$ and $(A',C')$ need not be defined in the same probability space.  Nonetheless, having them coupled sometimes simplifies notations and arguments.
\begin{lemma}[Conditioning on common information / Convexity of relative entropy]
\label{lem:relative-entropy:conditioning-on-common-info}
	%Let $A$ and $A'$ be random variables taking values in a finite set, and let $C$ be another random variable.
    Let $A$, $A'$ and $C$ be discrete random variables.
	Then, 
	\begin{equation}
		D(A\relto A') \leq D\big((A\given C)\relto[\big](A'\given C)\big) \;.
	\end{equation}
\end{lemma}
\begin{proof}
	Breaking down $D\big((A,C)\relto[\big](A',C)\big)$ using the chain rule in two different ways gives
	\begin{equation}
		D(A\relto A') + D\big((C\given A)\relto[\big](C\given A')\big) =
		D(C\relto C) + D\big((A\given C)\relto[\big](A'\given C)\big) \;.
	\end{equation}
	Since $D(C\relto C)=0$ and $D\big((C\given A)\relto[\big](C\given A')\big)\geq 0$, we obtain the claimed inequality.
\end{proof}

We also define
\begin{align}
    D\big((A\given C) \relto[\big] A'\big) \isdef
        \sum_{c\in\Gamma}\sum_{a\in\Sigma}\PP(A=a,C=c) \log\frac{\PP(A=a\given C=c)}{\PP(A'=a)} \;.
\end{align}
It is easy to verify that:
\begin{observation}[Conditional relative entropy and mutual information] %[Csiszár identity]
\label{obs:csiszar-identity}
	% Let $A$ and $A'$ be random variables taking values in a finite set, and let $C$ be another random variable.
	% Then, 
	\begin{equation}
		D\big((A\given C) \relto[\big] A'\big) = D(A \relto A') + I(A : C) \;.
	\end{equation}
\end{observation}
% \begin{proof}
% 	We have
% 	\begin{align}
% 		D\big((A\given C) \relto[\big] A'\big) &=
% 			\sum_{a,c}\PP(A=a,C=c)\log\frac{\PP(A=a\given C=c)}{\PP(A'=a)} \\
% 		&=
% 			\sum_a\PP(A=a)\log\frac{1}{\PP(A'=a)} - H(A\given C) + H(A) - H(A) \\
% 		&=
% 			\sum_a\PP(A=a)\log\frac{\PP(A=a)}{\PP(A'=a)} + I(A:C) \\
% 		&=
% 			D(A \relto A') + I(A:C) \;.
%             \qedhere
% 	\end{align}
% \end{proof}

We will need the following upper bound.
\begin{lemma}[Upper bound on relative entropy]
\label{lem:relative-entropy:bound}
	Let $\lambda=\lambda_q$  be a Bernoulli measure with strictly positive marginal $q\colon\Sigma\to(0,1)$.
	Let $X$ and $Z$ be random configurations from $\Sigma^{\ZZ^d}$, and suppose that $Z\sim\lambda$.
	Then, %for every two disjoint finite sets $A,B\Subset\ZZ^d$, we have
	\begin{equation}
		D\big((X_A\given X_B) \relto[\big] Z_A\big) \leq
			\abs{A}\log(\sfrac{1}{q_\smallest}) \;.
	\end{equation}
	for every two disjoint finite sets $A,B\Subset\ZZ^d$,
	where $q_\smallest\isdef\min\{q(a): a\in\Sigma\}$.
	%\comment{Better upper bound based on variational principle?}
\end{lemma}
\begin{proof}
	First, observe that for every probability distribution $p\colon\Sigma\to[0,1]$, we have
	\begin{align}
		D(p\relto q) &= \sum_{a\in\Sigma}p(a)\log\frac{p(a)}{q(a)}
			= \sum_{a\in\Sigma}p(a)\log\frac{1}{q(a)} - H(p)
			\leq \sum_{a\in\Sigma}p(a)\log(\sfrac{1}{q_\smallest})
			= \log(\sfrac{1}{q_\smallest}) \;.
	\end{align}
	Now, let $k_1,k_2,\ldots,k_n$ be an arbitrary ordering of the elements of~$A$,
	and for $i=1,2,\ldots,n$, let $A_{<i}\isdef\{k_1,k_2,\ldots,k_{i-1}\}$.
	Using the chain rule, we can write
	\begin{align}
		D\big((X_A\given X_B) \relto[\big] (Z_A\given Z_B)\big) &=
			\sum_{i=1}^n D\big((X_{k_i}\given X_{B\cup A_{<i}}) \relto[\big] (Z_{k_i}\given Z_{B\cup A_{<i}})\big)
	\end{align}
	Since $Z$ is \ac{iid} with marginal~$q$, the above observation and averaging give us
	\begin{align}
		D\big((X_{k_i}\given X_{B\cup A_{<i}}) \relto[\big] Z_{k_i}\big)
		&\leq \log(\sfrac{1}{q_\smallest})
	\end{align}
	for each $i$. The claim follows.
\end{proof}

When $\mu$ and $\lambda$ are probability measures on $\Sigma^{\ZZ^d}$ and $J\Subset\ZZ^d$, we use the notation $D_J(\mu\relto\lambda)$ for the relative entropy $D(\mu_J\relto\lambda_J)$ between the marginals of $\mu$ and $\lambda$ on~$J$.

\section{Entropy contraction for non-interacting Markov chains}
\label{sec:non-interacting-MCs}
In this section, we present relative entropy contraction inequalities (also known as strong data processing inequalities) for collections of non-interacting finite-state Markov chains that evolve either synchronously or asynchronously.  The results discussed here are either standard themselves, are simple variations of standard results, or easily derived from known results (see \ac{eg}, \cite{Raginsky2016,CMS2025}).

We start by recalling the data processing inequalities for a single Markov chain.

\begin{proposition}[Weak data processing inequality]
\label{prop:WDPI}
	Let $\Sigma$ and $\Gamma$ be finite sets.
	Let $\theta\colon\Sigma\times\Gamma\to[0,1]$ a stochastic matrix.
	Then, $D(p\theta \relto q\theta)\leq D(p \relto q)$ for every two probability distributions $p,q\colon\Sigma\to[0,1]$.
\end{proposition}
The proof is a straightforward application of the chain rule for relative entropy.
%\begin{proof}[Proof \comment{temporary; to be removed}]
%\end{proof}
%The following, which is more often referred to as the data processing inequality, is a consequence of Proposition~\ref{prop:WDPI}.
The following special case of Proposition~\ref{prop:WDPI} explains its namesake.
\begin{proposition}[Weak data processing inequality for mutual information]
\label{prop:WDPI:mutual-information}
	Let $A, B, C$ be random variables such that $A\to B\to C$.
	Then, $I(A:C)\leq I(A:B)$.
\end{proposition}

\begin{proposition}[Strong data processing inequality]
\label{prop:SDPI}
	Let $\Sigma$ be finite set.
	Let $\theta\colon\Sigma\times\Sigma\to(0,1)$ be a strictly positive stochastic matrix with stationary distribution~$q\colon\Sigma\to(0,1)$.  Let $\kappa>0$ be such that $\theta(a,b)\geq\kappa q(b)$ for each $a,b\in\Sigma$.
	Then, $D(p\theta \relto q)\leq (1-\kappa) D(p \relto q)$ for every probability distribution $p\colon\Sigma\to[0,1]$.
\end{proposition}
The contraction factor $1-\kappa$ is not the sharpest possible, even in the above general setting.  See for instance~\cite{Raginsky2016,CMS2025,CCG+2025}.  We include a proof of the above version because it illustrates the starting idea in the proofs of Theorems~\ref{thm:main:PCA:ergodicity} and~\ref{thm:main:IPS:ergodicity}.
\begin{proof}[Proof of Proposition~\ref{prop:SDPI}]
	By virtue of the hypothesis, we can decompose $\theta$ as
	\begin{equation}
		\theta(a,b) = \kappa q(b) + (1-\kappa)\widetilde{\theta}(a,b) \;,
	\end{equation}
	where $\widetilde{\theta}\colon\Sigma\times\Sigma\to[0,1]$ is another stochastic matrix. %, which may not be strictly positive.
	Let us construct random variables $A$, $\widetilde{B}$, $B$, $Q$ and $E$ as follows.
	We draw $A$ according to~$p$ and $Q$ according to $q$ independently of one another.  We then draw $\widetilde{B}$ according to~$\widetilde{\theta}(A,\cdot)$, independently of~$Q$.  We also draw $E$ according to the Bernoulli distribution with parameter~$\kappa$, independently of $A$, $Q$, and $\widetilde{B}$.  Lastly, we let
	\begin{equation}
		B =
			\begin{cases}
				Q					& \text{if $E=\symb{1}$,}\\
				\widetilde{B}		& \text{if $E=\symb{0}$.}
			\end{cases}
	\end{equation}
	Observe that $A\markovto[\theta]B$, and in particular, $B$ has distribution~$p\theta$.
	We have
	\begin{align}
		D(p\theta\relto q) &=
			D(B\relto Q) \\
		&\leq
			D\big((B\given E)\relto[\big](Q\given E)\big)
				&& \text{(Lemma~\ref{lem:relative-entropy:conditioning-on-common-info})} \\
		&=
			D\big((B\given E)\relto[\big] Q\big)
				&& \text{(independence of $Q$ and $E$)} \\
		&=
			\begin{multlined}[t]
				\PP(E=\symb{1})D\big((B\given E=\symb{1})\relto[\big] Q\big) \\
				+
				\PP(E=\symb{0})D\big((B\given E=\symb{0})\relto[\big] Q\big)
			\end{multlined} \\
		&=
			\kappa D(Q\relto Q)
			+
			(1-\kappa) D(\widetilde{B}\relto Q)
				&& \text{(definition of~$B$)} \\
		&\leq
			(1-\kappa) D(A\relto Q)
				&& \text{(Proposition~\ref{prop:WDPI})} \\
		&=
			(1-\kappa) D(p\relto q) \;,
	\end{align}
	which proves the claim.
\end{proof}

We next consider a collection of non-interacting identical Markov chains that evolve synchronously or asynchronously.  More specifically, given a stochastic matrix $\theta\colon\Sigma\times\Sigma\to[0,1]$ and a positive integer~$n$, we consider two stochastic matrices $\theta_n,\widehat{\theta}_n\colon\Sigma^n\times\Sigma^n\to[0,1]$, where
\begin{align}
	\theta_n(\vect{a},\vect{b}) &\isdef \prod_{i=1}^n\theta(a_i,b_i) & &\text{and}&
	\widehat{\theta}_n(\vect{a},\vect{b}) &\isdef
		\frac{1}{n}\sum_{i=1}^n \theta(a_i,b_i)\prod_{j\neq i}\indicator{a_j}(b_j)
\end{align}
for $\vect{a}=(a_1,a_2,\ldots,a_n),\vect{b}=(b_1,b_2,\ldots,b_n)\in\Sigma^n$.
In words, $\theta_n$ is the transition matrix of a Markov chain involving $n$ components that are updated independently, in parallel, according to the transition probabilities prescribed by~$\theta$.  In contrast, in a Markov chain with transition matrix $\widehat{\theta}_n$, at every step, only one of the $n$ components is selected uniformly at random and is updated according to~$\theta$.  Observe that if $q$ is a stationary distribution for~$\theta$, then $\big({\otimes_{i=1}^n q}\big)(\vect{a})\isdef\prod_{i=1}^n q(a_i)$ is a stationary distribution for both $\theta_n$ and $\widehat{\theta}_n$.

\begin{proposition}[Strong data processing inequality for synchronous updating]
\label{prop:SDPI:synchronous}
	Let $\Sigma$ be a finite set.
	Let $\theta\colon\Sigma\times\Sigma\to(0,1)$ be a strictly positive stochastic matrix with stationary distribution~$q\colon\Sigma\to(0,1)$.  Let $\kappa>0$ be such that $\theta(a,b)\geq\kappa q(b)$ for each $a,b\in\Sigma$.
	Let $n$ be a positive integer.
	Then,
	\begin{equation}
		D\left(\vect{p}\theta_n \relto[\big] \otimes_{i=1}^n q\right) \leq (1-\kappa) D\big(\vect{p} \relto[\big] \otimes_{i=1}^n q\big)
	\end{equation}
	for every probability distribution $\vect{p}\colon\Sigma^n\to[0,1]$.
\end{proposition}
\begin{proof}
	Consider random variables $\vect{A}=(A_1,A_2,\ldots,A_n)$ and $\vect{B}=(B_1,B_2,\ldots,B_n)$ generated by first drawing $\vect{A}$ according to~$\vect{p}$ and then, for each $i$, drawing $B_i$ according to $\theta(A_k,\cdot)$, independently of one another.  Note that $\vect{A}\markovto[\theta_n]\vect{B}$.
	Consider also a collection $\vect{Q}=(Q_1,Q_2,\ldots,Q_n)$ of \ac{iid} random variables, each distributed according to~$q$.
%	Consider random variables $\vect{A}=(A_1,A_2,\ldots,A_n)$, $\vect{B}=(B_1,B_2,\ldots,B_n)$, and $\vect{Q}=(Q_1,Q_2,\ldots,Q_n)$ such that
%	\begin{itemize}
%		\item $\vect{A}$ is distributed according to~$\vect{p}$,
%		\item $\vect{A}\markovto[\theta_n]\vect{B}$,
%			in particular, given $\vect{A}$, the random variables $B_1,B_2,\ldots,B_n$ are independent with $A_i\markovto[\theta] B_i$ for each~$i$.
%		\item $\vect{Q}\sim\otimes_{i=1}^n q$, that is, $Q_i$'s are \ac{iid}, each distributed according to~$q$,
%	\end{itemize}
%	$\vect{A}$ is distributed according to~$\vect{p}$,
%	$\vect{A}\markovto[\theta_n]\vect{B}$, and
%	$\vect{Q}$ is distributed according to~$\otimes_{i=1}^n q$.

	Let us use the notation
	\useshortskip
	\begin{align}
		\vect{A}_{<k} &\isdef (A_1,A_2,\ldots,A_{k-1}) \;.
	\end{align}
	Note that, for every $k$,
	\begin{align}
		\MoveEqLeft
		D\big((B_k\given \vect{B}_{<k}) \relto[\big] Q_k\big) \\
%		&=
%			D(B_k\relto Q_k) + I\big(B_k:(B_1,B_2,\ldots,B_{k-1})\big)
%				&& \text{(Observation~\ref{obs:csiszar-identity})} \\
%		&\leq
%			D(B_k\relto Q_k) + I\big(B_k:(A_1,A_2,\ldots,A_{k-1})\big)
%				&& \text{(Proposition~\ref{prop:WDPI:mutual-information})} \\
%		&=
		&\leq
			D\big((B_k\given\vect{A}_{<k}) \relto[\big] Q_k\big)
				&& \text{(Observation~\ref{obs:csiszar-identity} + Proposition~\ref{prop:WDPI:mutual-information})} \\
		&\leq
			(1-\kappa)D\big((A_k\given\vect{A}_{<k}) \relto[\big] Q_k\big) \;,
				&& \text{(Proposition~\ref{prop:SDPI})}
		\label{eq:SDPI:synchronous:proof:chain-term}
	\end{align}
	where for the second inequality, we have used the fact that, conditioned on $\vect{A}_{<k}$, we still have $A_k\markovto[\theta]B_k$.
	It follows that
	\begin{align}
		D\big(\vect{p}\theta_n \relto[\big] \otimes_{i=1}^n q\big)
		&=
			D(\vect{B} \relto[\big] \vect{Q}) \\
		&=
			\sum_{k=1}^n D\big((B_k\given\vect{B}_{<k}) \relto[\big] (Q_k\given\vect{Q}_{<k})\big)
				&& \text{(chain rule)} \\
%		&=
%			\sum_{k=1}^n D\big((B_k\given B_1,B_2,\ldots,B_{k-1}) \relto[\big] Q_k\big)
%				&& \text{(independence of $Q_i$'s)} \\
%		&\leq
%			\sum_{k=1}^n (1-\kappa)D\big((A_k\given A_1,A_2,\ldots,A_{k-1}) \relto[\big] Q_k\big)
%				&& \text{(by~\eqref{eq:SDPI:synchronous:proof:chain-term})} \\
%		&=
		&\leq
			(1-\kappa)
			\sum_{k=1}^n D\big((A_k\given\vect{A}_{<k}) \relto[\big] (Q_k\given\vect{Q}_{<k})\big)
				&& \text{(independence of $Q_i$'s + \eqref{eq:SDPI:synchronous:proof:chain-term})} \\
		&=
			(1-\kappa)
			D(\vect{A} \relto[\big] \vect{Q})
				&& \text{(chain rule)} \\
		&=
			(1-\kappa) D\big(\vect{p} \relto[\big] \otimes_{i=1}^n q\big) \;,
	\end{align}
	which proves the claim.
\end{proof}

\begin{proposition}[Strong data processing inequality for asynchronous updating]
\label{prop:SDPI:asynchronous}
	Let $\Sigma$ be a finite set.
	Let $\theta\colon\Sigma\times\Sigma\to(0,1)$ be a strictly positive stochastic matrix with stationary distribution~$q\colon\Sigma\to(0,1)$.  Let $\kappa>0$ be such that $\theta(a,b)\geq\kappa q(b)$ for each $a,b\in\Sigma$.
	Let $n$ be a positive integer.
	Then,
	\begin{equation}
		D\big(\vect{p}\widehat{\theta}_n \relto[\big] \otimes_{i=1}^n q\big) \leq (1-\sfrac{\kappa}{n}) D\big(\vect{p} \relto[\big] \otimes_{i=1}^n q\big)
	\end{equation}
	for every probability distribution $\vect{p}\colon\Sigma^n\to[0,1]$.
\end{proposition}
\begin{proof}
	Consider random variables $\vect{A}=(A_1,A_2,\ldots,A_n)$, $\vect{B}=(B_1,B_2,\ldots,B_n)$, $\vect{C}=(C_1,C_2,\ldots,C_n)$, and $K$ generated as follows.  We first draw $\vect{A}$ according to $\vect{p}$ and $K$ uniformly from $\{1,2,\ldots,n\}$ independently of each other.  For each $i$, we then draw $B_i$ according to $\theta(A_i,\cdot)$ independently of one another, and let
	\begin{equation}
		C_i \isdef
			\begin{cases}
				B_i			& \text{if $K=i$,}\\
				A_i			& \text{otherwise.}
			\end{cases}
	\end{equation}
	Clearly, $\vect{A}\markovto[\widehat{\theta}_n]\vect{C}$.
	We also draw a collection $\vect{Q}=(Q_1,Q_2,\ldots,Q_n)$ of samples from $q$ that are independent of one another and of~$K$.
	
%	First, let us verify that
%	\begin{equation}
%		D(\vect{C}\relto\vect{Q}) \leq D\big((\vect{C}\given K)\relto[\big](\vect{Q}\given K)\big)
%		\label{eq:SDPI:asynchronous:proof:condition-on-K}
%	\end{equation}
%	Breaking down $D\big((\vect{C},K)\relto[\big](\vect{Q},K)\big)$ using the chain rule in two different ways gives
%	\begin{equation}
%		D(\vect{C}\relto\vect{Q}) + D\big((K\given\vect{C})\relto[\big](K\given\vect{Q})\big) =
%		D(K\relto K) + D\big((\vect{C}\given K)\relto[\big](\vect{Q}\given K)\big) \;.
%	\end{equation}
%	Since $D(K\relto K)=0$ and $D\big((K\given\vect{C})\relto[\big](K\given\vect{Q})\big)\geq 0$, we obtain the claimed inequality.
	
	Let us use the notation
	\useshortskip
	\begin{align}
		\vect{A}_{\neq i} &\isdef (A_1,\ldots,A_{i-1},A_{i+1},\ldots,A_n) \;, \\
		\vect{A}_{< i} &\isdef (A_1,A_2,\ldots,A_{i-1}) \;.
	\end{align}
	Note that, for each $i$,
	\begin{align}
		\MoveEqLeft
		D\big((A_1,\ldots,A_{i-1},B_i,A_{i+1},\ldots,A_n)\relto[\big]\vect{Q}\big) \\
		&=
			D\big(\vect{A}_{\neq i}\relto[\big]\vect{Q}_{\neq i}\big)
			+
			D\big((B_i\given\vect{A}_{\neq i})\relto[\big](Q_i\given\vect{Q}_{\neq i})\big)
				&& \text{(chain rule)} \\
		&=
			D\big(\vect{A}_{\neq i}\relto[\big]\vect{Q}_{\neq i}\big)
			+
			D\big((B_i\given\vect{A}_{\neq i})\relto[\big] Q_i\big)
				&& \text{(independence of $Q_i$'s)} \\
		&\leq
			D\big(\vect{A}_{\neq i}\relto[\big]\vect{Q}_{\neq i}\big)
			+
			(1-\kappa)D\big((A_i\given\vect{A}_{\neq i})\relto[\big] Q_i\big) \;,
				&& \text{(Proposition~\ref{prop:SDPI})} \\
%		&=
%			\begin{multlined}[t]
%				D\big(\vect{A}_{\neq i}\relto[\big]\vect{Q}_{\neq i}\big)
%				+
%				D\big((A_i\given\vect{A}_{\neq i})\relto[\big](Q_i\given\vect{Q}_{\neq i})\big) \\
%				-
%				\kappa D\big((A_i\given\vect{A}_{\neq i})\relto[\big] Q_i\big)
%			\end{multlined}
%				&& \text{(independence of $Q_i$'s)} \\
		&=
			D(\vect{A}\relto[\big]\vect{Q}) - \kappa D\big((A_i\given\vect{A}_{\neq i})\relto[\big] Q_i\big)
				&& \text{(chain rule + independence of $Q_i$'s)}
		\label{eq:SDPI:asynchronous:proof:chain-term}
%		\\
%		&=
%			D(\vect{A}\relto[\big]\vect{Q})
%			- \kappa\Big[D(A_i\relto Q_i) + I\big(A_i:\vect{A}_{\neq i}\big)\Big]
%				&& \text{(Observation~\ref{obs:csiszar-identity})} \\
%		&\leq
%			D(\vect{A}\relto[\big]\vect{Q})
%			- \kappa\Big[D(A_i\relto Q_i) + I\big(A_i:\vect{A}_{<i}\big)\Big]
%				&& \text{(Proposition~\ref{prop:WDPI:mutual-information})} \\
%		&=
%			D(\vect{A}\relto[\big]\vect{Q}) - \kappa D\big((A_i\given\vect{A}_{<i})\relto[\big] Q_i\big) \;,
%				&& \text{(Observation~\ref{obs:csiszar-identity})}
	\end{align}
	where for the inequality, we have used the fact that, conditioned on $\vect{A}_{\neq i}$, we still have $A_i\markovto[\theta]B_i$.
	
	We can now write
	\begin{align}
		D\big(\vect{p}\widehat{\theta}_n \relto[\big] \otimes_{i=1}^n q\big)
		&=
			D(\vect{C}\relto\vect{Q}) \\
		&\leq
			D\big((\vect{C}\given K)\relto[\big](\vect{Q}\given K)\big)
				%&& \text{(by~\eqref{eq:SDPI:asynchronous:proof:condition-on-K})} \\
				&& \text{(Lemma~\ref{lem:relative-entropy:conditioning-on-common-info})} \\
		&=
			\sum_{i=1}^n\PP(K=i)D\big((\vect{C}\given K=i)\relto[\big](\vect{Q}\given K=i)\big) \\
		&=
			\frac{1}{n}\sum_{i=1}^n D\big((A_1,\ldots,A_{i-1},B_i,A_{i+1},\ldots,A_n)\relto[\big]\vect{Q}\big)
				&& \text{(definition of $C_i$'s)} \\
		&\leq
			\frac{1}{n}\sum_{i=1}^n
			\Big[
				D(\vect{A}\relto\vect{Q}) - \kappa D\big((A_i\given\vect{A}_{\neq i})\relto[\big] Q_i\big)
			\Big]
				&& \text{(by~\eqref{eq:SDPI:asynchronous:proof:chain-term})} \\
		&\leq
			\frac{1}{n}\sum_{i=1}^n
			\Big[
				D(\vect{A}\relto\vect{Q}) - \kappa D\big((A_i\given\vect{A}_{<i})\relto[\big] Q_i\big)
			\Big]
				&& \text{(Observation~\ref{obs:csiszar-identity} + Proposition~\ref{prop:WDPI:mutual-information})} \\
%		&=
%			D(\vect{A}\relto\vect{Q})
%			- \frac{\kappa}{n}\sum_{i=1}^n D\big((A_i\given\vect{A}_{<i})\relto[\big] Q_i\big) \\
		&=
			D(\vect{A}\relto\vect{Q})
			- \frac{\kappa}{n}D(\vect{A}\relto\vect{Q})
				&& \text{(chain rule + independence of $Q_i$'s)} \\
		&=
			(1-\sfrac{\kappa}{n})D\big(\vect{p} \relto[\big] \otimes_{i=1}^n q\big) \;,
	\end{align}
	proving the proposition.
\end{proof}

\section{PCA with stationary Bernoulli measures}
\label{sec:PCA:ergodicity}

\subsection{Noise decomposition}    % Or: Perturbative representation
\label{sec:PCA:noise-decomposition}

Let $\Phi$ be a $d$-dimensional PCA with a strictly positive local transition rule $\varphi\colon\Sigma^N\times\Sigma\to(0,1)$.
Let $\lambda_q=\otimes_{i\in\ZZ^d}q$ be a Bernoulli measure with marginal~$q$, and suppose that $\lambda_q$ is stationary under~$\Phi$.

The starting idea in the proof of Theorem~\ref{thm:main:PCA:ergodicity}, inspired by earlier results on random perturbations of deterministic cellular automata~\cite{MST2019,Taati2021}, is to represent $\Phi$ as a perturbation of another PCA with a zero-range, memoryless noise, in such a way that the new PCA and the noise both preserve~$\lambda_q$.  

Let $\kappa>0$ be such that $\varphi(a_N,b)\geq\kappa q(b)$ for every $a_N\in\Sigma^N$ and $b\in\Sigma$.  As in the proof of Proposition~\ref{prop:SDPI}, we can decompose $\varphi$ as
\begin{equation}
	\varphi(a_N,b) = \kappa q(b) + (1-\kappa)\psi(a_N,b) \;,
\end{equation}
where $\psi\colon\Sigma\times\Sigma\to[0,1]$ is another local transition rule.
This decomposition can be interpreted as follows:
In order to draw a sample from $\varphi(a_N,\cdot)$, we can first flip a coin with parameter~$\kappa$.  If the coin comes up heads, we draw a sample from~$q$; %(independently of $a_N$ and the coin);
otherwise, we draw a sample from $\psi(a_N,\cdot)$.
Alternatively, we can reinterpret this as first taking a sample from $\psi(a_N,\cdot)$ and then subjecting the result to zero-range, memoryless noise with error probability $\kappa$ and replacement distribution~$q$.

More specifically, let $\Psi$ be the PCA with local transition rule~$\psi$, and let $\Theta_{\kappa,q}$ denote the zero-range PCA with local transition rule $\theta\colon\Sigma\times\Sigma\to[0,1]$ defined by $\theta(a,b)\isdef(1-\kappa)\indicator{a}(b)+\kappa q(b)$.  We refer to $\Theta_{\kappa,q}$ as the \emph{zero-range, memoryless noise} with \emph{error probability} $\kappa$ and \emph{replacement distribution}~$q$.

\begin{observation}[Noise decomposition]
\label{obs:PCA:noise-decomposition}
	$\Phi=\Psi\Theta_{\kappa,q}$.
\end{observation}

% Clearly, the noise component $\Theta_{\kappa,q}$ preserves the Bernoulli measure~$\lambda_q$.
% \begin{lemma}
% \label{lem:PCA:decomposition:invariance}
% 	If $\Phi$ preserves the Bernoulli measure~$\lambda_q$ and $\kappa<1$, then so does~$\Psi$.
% \end{lemma}
% \begin{proof}
% 	\comment{To be added!}
% \end{proof}

The kernel $\Theta_{\kappa,q}$ acts injectively on probability measures.
\begin{lemma}[Injectivity of noise kernel]
\label{lem:noise-kernel:injectivity}
    %Let $q\colon\Sigma\to[0,1]$ be a probability distribution and $0<\kappa<1$.
    If $\kappa<1$, the map $\mu\mapsto\mu\Theta_{\kappa,q}$ is one-to-one on $\family{P}(\Sigma^{\ZZ^d})$.
\end{lemma}
\begin{proof}
    The noise matrix can be written as $\theta=(1-\kappa)I+\kappa Q$, where $I$ is the identity matrix with rows and columns indexed by~$\Sigma$, and $Q$ is the $\Sigma\times\Sigma$ matrix with vector~$q$ on each row.
    Observe that $Q$ is idempotent, that is, $Q^2=Q$.  One can now verify that $\theta$ is invertible with inverse
    \begin{equation}
        \theta^{-1} = \frac{1}{1-\kappa}I-\frac{\kappa}{(1-\kappa)}Q \;.
    \end{equation}
    Since any tensor product of invertible matrices is again invertible, for any finite region $A\Subset\ZZ^2$, the noise matrix $\theta_A(u,v)\isdef\prod_{i\in A}\theta(u_i,v_i)$ induced by $\theta$ on $\Sigma^A$ is again invertible.
    %It follows that the map $\mu\mapsto\mu\Theta_{\kappa,q}$ on $\family{P}(\Sigma^{\ZZ^2})$ is one-to-one.
    Now, let $\mu_1,\mu_2\in\family{P}(\Sigma^{\ZZ^d})$ be such that $\mu_1\Theta_{\kappa,q}=\mu_2\Theta_{\kappa,q}$.
    Then, $\mu_1$ and $\mu_2$ agree on every cylinder, hence $\mu_1=\mu_2$.
\end{proof}

From the latter lemma, and the fact that the noise $\Theta_{\kappa,q}$ preserves the Bernoulli measure $\lambda_q$, it immediately follows that:
\begin{proposition}[Stationarity for noise decomposition]
\label{prop:PCA:decomposition:invariance}
    Suppose $\Phi=\Psi\Theta_{\kappa,q}$ for some kernel $\Psi$ and some $\kappa<1$.
	If $\Phi$ preserves the Bernoulli measure~$\lambda_q$, then so does~$\Psi$.
\end{proposition}

\subsection{Evolution of relative entropy}
\label{sec:PCA:entropy-evolution}
In this section, we provide a simple proof of Theorem~\ref{thm:main:PCA:ergodicity}.
In view of the noise decomposition of~$\Phi$ (Observation~\ref{obs:PCA:noise-decomposition} and Proposition~\ref{prop:PCA:decomposition:invariance}), we start by examining the effects of each component on relative entropy separately.

\begin{lemma}[Local diffusion of entropy] %[Effect of a PCA on entropy]
\label{lem:PCA:entropy-diffusion}
	%Let $\Phi$ be a $d$-dimensional PCA with local transition rule $\varphi\colon\Sigma^N\times\Sigma\to[0,1]$,
	Let $\Psi$ be a PCA with dependence neighbourhood $N\Subset\ZZ^d$ where $0\in N$,
	and suppose that $\Psi$ preserves a Bernoulli measure~$\lambda=\lambda_q$.
	Then,
%	\begin{align}
%		D\big((\mu\Phi)_J\relto[\big]\lambda_J\big) &\leq
%			D\big(\mu_{N(J)}\relto[\big]\lambda_{N(J)}\big) \\
%		&\leq
%			D\big(\mu_J\relto[\big]\lambda_J\big) + \abs{\partial N(J)}\log(\sfrac{1}{q_\smallest})
%	\end{align}
	\begin{equation}
		D_J(\mu\Psi\relto\lambda) \leq
			D_{N(J)}(\mu\relto\lambda)
			\label{eq:PCA:entropy-diffusion:1}
			%\tag{\sun:1}\\
		% &\leq
		% 	D_J(\mu\relto\lambda) + \abs{\partial N(J)}\log(\sfrac{1}{q_\smallest})
		% 	\label{eq:PCA:entropy-diffusion:2}
		% 	\tag{\sun:2}
	\end{equation}
	for every probability measure~$\mu$ and every finite set $J\subseteq\ZZ^d$,
	where $q_\smallest\isdef\min\{q(a): a\in\Sigma\}$.
\end{lemma}
\begin{proof}
    This is an application of the weak data processing inequality~(Proposition~\ref{prop:WDPI}).
\end{proof}
% \begin{proof}
% 	The first inequality is an application of the weak data processing inequality~(Proposition~\ref{prop:WDPI}).
% %	To verify the second inequality, consider random variables $X_{N(J)}\sim\mu_{N(J)}$ and $Z_{N(J)}\sim\lambda_{N(J)}$.
% %	
% %	Let $k_1,k_2,\ldots,k_m$ be an enumeration of the elements of $\partial N(J)$.
% %	Then, by the chain rule and the independence of $Z_k$'s,
% %	\begin{align}
% %		%\MoveEqLeft
% %		%D\big(\mu_{N(J)}\relto[\big]\lambda_{N(J)}\big)
% %		D_{N(J)}(\mu\relto\lambda)
% %		&=
% %			D\big(X_{N(J)}\relto[\big]Z_{N(J)}\big) \\
% %		&=
% %			D\big(X_J\relto[\big]Z_J\big)
% %			+ \sum_{i=1}^m D\big((X_J,X_{k_i}\given X_{k_1},X_{k_2},\ldots,X_{k_{i-1}})\relto[\big]Z_{k_i}\big)
% %%				&& \text{(chain rule)} \\
% %%		&\leq
% %%			D\big(X_J\relto[\big]Z_J\big)
% %%			+ \sum_{i=1}^m \log\frac{1}{q_{\smallest}}
% %	\end{align}
% %	But since $Z_{k_i}\sim q$, we have
% %	\begin{equation}
% %		D\big((X_{k_i}\given X_{k_1},X_{k_2},\ldots,X_{k_{i-1}})\relto[\big]Z_{k_i}\big)
% %			\leq \log(\sfrac{1}{q_\smallest}) \;.
% %	\end{equation}
% %	The second inequality now follows.
% 	The second inequality follows from the chain rule and Lemma~\ref{lem:relative-entropy:bound}. \irene{Is the second line useful for what comes next? If yes, specify the notation $\partial N(J)$? And add one more line in the inequalities to separate the chain rule and Lemma 2.5?}
% \end{proof}

\begin{lemma}[Entropy decay]  %[Effect of noise on entropy]
\label{lem:noise:contraction}
	Let $\Theta_{\kappa,q}$ be the global kernel of a memoryless noise
	with error probability~$\kappa>0$ and replacement distribution~$q$,
	and let $\lambda=\lambda_q$ be the Bernoulli measure with marginal~$q$.
	Then, %for every finite set $J\Subset\ZZ^d$, we have
	\begin{equation}
		D_J(\mu\Theta_{\kappa,q},\lambda)
		\leq
			(1-\kappa) D_J(\mu,\lambda)
	\end{equation}
	for every probability measure~$\mu$ and every finite set $J\subseteq\ZZ^d$.
\end{lemma}
\begin{proof}
	This is an application of the strong data processing inequality for non-interacting Markov chains with synchronous updating (Proposition~\ref{prop:SDPI:synchronous}).
\end{proof}

Putting Observation~\ref{obs:PCA:noise-decomposition} and Proposition~\ref{prop:PCA:decomposition:invariance}, \eqref{eq:PCA:entropy-diffusion:1} in Lemma~\ref{lem:PCA:entropy-diffusion}, and Lemma~\ref{lem:noise:contraction} together gives the following.
\begin{proposition}[Local entropy diffusion with decay] %[Effect of a positive-rate PCA on entropy]
\label{prop:PCA:relative-entropy:step}
	Let $\Phi$ be a PCA with a strictly positive local transition rule $\varphi\colon\Sigma^N\times\Sigma\to(0,1)$, where $N\ni 0$.
	Suppose that $\Phi$ admits a stationary Bernoulli measure $\lambda=\lambda_q$.
	Then,
	\begin{equation}
		D_J(\mu\Phi\relto\lambda) \leq
			(1-\kappa)D_{N(J)}(\mu\relto\lambda)
	\end{equation}
	for every probability measure~$\mu$ and every finite set $J\subseteq\ZZ^d$, where
	$\kappa \isdef \min\big\{\varphi(a_N,b)/q(b): a_N\in\Sigma^N,b\in\Sigma\big\}$ and
	%$c_1\isdef (1-\kappa)\log(\sfrac{1}{q_\smallest})$, and
	$q_\smallest\isdef\min\{q(a): a\in\Sigma\}$.
\end{proposition}

We are now ready to prove the main result in discrete time.

\begin{proof}[Proof of Theorem~\ref{thm:main:PCA:ergodicity}] % (sightly weaker form)]
%	Let $\kappa$ and $q_\smallest$ be as in Proposition~\ref{prop:PCA:relative-entropy:step}.
%	Let
%	\begin{equation}
%		\kappa \isdef \min\big\{\varphi(a_N,b)/q(b): a_N\in\Sigma^N,b\in\Sigma\big\} \;.
%	\end{equation}
%	Then,
%	\begin{align}
%		D_J(\mu\Phi\relto\lambda)
%		&=
%			D_J\big(\mu\Psi\Theta_{\kappa,q}\relto[\big]\lambda\big)
%				&& \text{(Observation~\ref{obs:PCA:noise-decomposition})} \\
%		&\leq
%			(1-\kappa)D_J(\mu\Psi\relto\lambda)
%				&& \text{(Lemma~\ref{lem:noise:contraction})} \\
%		&\leq
%			(1-\kappa)D_{N(J)}(\mu\relto\lambda)
%				&& \text{(by~\eqref{eq:PCA:entropy-diffusion} in Lemma~\ref{lem:PCA:entropy-diffusion})}		
%	\end{align}
%	Iterating the latter inequality, we find that
	Iterating Proposition~\ref{prop:PCA:relative-entropy:step}, we find that
	\begin{equation}
		D_J(\mu\Phi^t\relto\lambda)
		\leq
			(1-\kappa)^t D_{N^t(J)}(\mu\relto\lambda) \;.
	\end{equation}
	for every $t\geq 0$.
	By Lemma~\ref{lem:relative-entropy:bound},
	\begin{equation}
		D_{N^t(J)}(\mu\relto\lambda) \leq \abs{N^t(J)}\log(\sfrac{1}{q_\smallest}) \;.
	\end{equation}
	Let $r\in\NN$ be the interaction radius of~$\Phi$, that is, the smallest integer such that $N\subseteq\zinterval{-r,r}^d$.
	Since $J$ has diameter~$n$, we can find $a\in\ZZ^d$ such that $J\subseteq J'\isdef a+[0,n-1]^d$.
	Therefore,
	\begin{equation}
		\abs{N^t(J)} \leq
			\abs{N^t(J')} =
			(n+2rt)^d \;.
	\end{equation}
	It follows that
	\begin{align}
		D_J(\mu\Phi^t\relto\lambda)
		&\leq
			(1-\kappa)^t(n+2rt)^d\log(\sfrac{1}{q_\smallest}) \\
%		&=
%			(1-\kappa)^t n^d \left(1+\frac{2r}{n}t\right)^d \\
		&\leq
			\alpha_1 \ee^{-\beta_1 t} n^d
	\end{align}
	for any $\beta_1$ satisfying $0<\beta_1<-\log(1-\kappa)$ and an appropriate choice of~$\alpha_1$.
	Pinsker's inequality (\ac{eg},~\cite[Lemma~11.6.1]{CT2006}) now gives
	\begin{equation}
		\norm{\mu\Phi^t-\lambda}_J \leq \alpha\ee^{-\beta t} n^{d/2} \;,
	\end{equation}
	where $\alpha\isdef\sqrt{\alpha_1/2}$ and $\beta\isdef\beta_1/2$.
\end{proof}

\begin{remark}
\label{rem:graph-with-sub-exponential-growth}
	The same argument as above shows that, if $\GG$ is a finitely generated group with sub-exponential growth, then any PCA on $\GG$ (as the lattice) that has positive transition probabilities and admits a stationary Bernoulli measure is ergodic.
	In fact, a similar ergodicity result holds for non-uniform PCA on any countable graph with sub-exponential growth, provided that the local rules at different sites are uniformly strictly positive.
\end{remark}

\section{IPS with stationary Bernoulli measures}
\label{sec:IPS:ergodicity}

\subsection{Change in entropy}

%\begin{proposition}[Change in entropy]
%\label{prop:IPS:entropy-change}
%	Let $\Phi$ be an IPS with generator~$L$ on $\Sigma^{\ZZ^d}$, and let $\lambda$ be a full-support probability measure that is stationary under~$\Phi$.  Let $\mu$ be a probability measure and $J\Subset\ZZ^d$.  Then,
%	$D_J(t)\isdef D_J(\mu\Phi^t\relto\lambda)$ is differentiable at every $t>0$ with
%	\begin{equation}
%		\dot{D}_J(t) =
%			\sum_{w\in\Sigma^J} (\mu\Phi^t L)([w])\log\frac{(\mu\Phi^t)([w])}{\lambda([w])} \;.
%	\end{equation}
%	At $t=0$, the same expression is valid unless $\mu([w])=0$ for some $w\in\Sigma^J$, in which case $\dot{D}_J(0)=-\infty$.
%\end{proposition}
\begin{proposition}[Entropy change in terms of generator]
\label{prop:IPS:entropy-change:in-terms-of-generator}
	Let $\Phi$ be an IPS with generator~$L$ on $\Sigma^{\ZZ^d}$, and let $\lambda$ be a full-support probability measure that is stationary under~$\Phi$.  Let $\mu$ be a probability measure and $J\Subset\ZZ^d$.
	Let $\mu^t\isdef\mu\Phi^t$ and $D_J(t)\isdef D_J(\mu^t\relto\lambda)$.
	Then, at every $t\geq 0$, $D_J(t)$ is differentiable with
	\begin{equation}
		\dot{D}_J(t) =
			\sum_{w\in\Sigma^J} (\mu^t L)([w])\log\frac{\mu^t([w])}{\lambda([w])}
	\end{equation}
	unless $\mu^t([w])=0$ for some $w\in\Sigma^J$.
	%, in which case $\dot{D}_J(0)=-\infty$.
\end{proposition}
\begin{proof} %\comment{Skip or move to the appendix?}
	For $c>0$, the function $g_c(x)\isdef x\log(\sfrac{x}{c})$ is differentiable at every $x>0$ with $g'_c(x)=\log(\sfrac{x}{c})+1$.
	
	Let $t\geq 0$	 and suppose $\mu^t([w])>0$ for all $w\in\Sigma^J$.
	Since $\mu^s=\mu\Phi^s$ is continuous in~$s$, we have $\mu^s([w])>0$ for all $w\in\Sigma^J$ in a neighbourhood of~$t$.  Therefore,
	\begin{align}
		\dot{D}_J(t) &=
			\frac{\dd}{\dd t}\sum_{w\in\Sigma^J} \mu^t([w])\log\frac{\mu^t([w])}{\lambda([w])} \\
		&=
			\sum_{w\in\Sigma^J} \bigg(\frac{\dd}{\dd t}\mu^t([w])\bigg)
				\bigg(\log\frac{\mu^t([w])}{\lambda([w])} + 1\bigg) \\
		&=
			\sum_{w\in\Sigma^J} (\mu^t L)([w])
				\bigg(\log\frac{\mu^t([w])}{\lambda([w])} + 1\bigg) \\
		&=
			\sum_{w\in\Sigma^J} (\mu^t L)([w])\log\frac{\mu^t([w])}{\lambda([w])}
			+
			(\mu^t L)(1) \\
		&=
			\sum_{w\in\Sigma^J} (\mu^t L)([w])\log\frac{\mu^t([w])}{\lambda([w])}
	\end{align}
	as claimed.
\end{proof}

%\begin{corollary}[Linearity of entropy change \comment{skip or clean up!}]
%	In the setting of Proposition~\ref{prop:IPS:entropy-change:in-terms-of-generator}, suppose $L=a_1 L_1 + a_2 L_2$, where $a_1,a_2\in\RR$ and $L_1,L_2$ are generators for two IPSs that preserve~$\lambda$.  Then,
%	\begin{equation}
%		\dot{D}_J(0) = a_1 \dot{D}_{1,J}(0) + a_2 \dot{D}_{2,J}(0)
%	\end{equation}
%	where $D_{1,J}(t)$ and $D_{2,J}(t)$ are the relative entropies associated with $L_1$ and $L_2$ respectively.
%\end{corollary}
%\comment{Warning!  The linearity may not hold at $t>0$ because the dependence on $\mu\Phi^t$ is not linear.}

Combining Propositions~\ref{prop:IPS:entropy-change:in-terms-of-generator} with Observation~\ref{obs:IPS:generator-in-terms-of-asynchronous} gives the following.
\begin{proposition}[Entropy change in terms of asynchronous updating]
\label{prop:IPS:entropy-change:in-terms-of-asynchronous}
	Let $\Phi$ be an IPS with a strictly positive local transition rule $\varphi\colon\Sigma^N\times\Sigma\to(0,1)$ and clock rate~$1$, and let $\lambda$ be a full-support probability measure that is stationary under~$\Phi$.  Let $\mu$ be a probability measure and $J\Subset\ZZ^d$.
	Let $\mu^t\isdef\mu\Phi^t$ and $D_J(t)\isdef D_J(\mu^t\relto\lambda)$.
	Then, $D_J(t)$ is differentiable at every $t>0$ and
	\begin{equation}
		\dot{D}_J(t) =
			\abs{J}\Big(
				D_J\big(\mu^t\widehat{\Phi}_J \relto[\big] \lambda \big)
				-
				D_J\big(\mu^t\widehat{\Phi}_J \relto[\big] \mu^t \big)
				-
				D_J(\mu^t \relto \lambda)
			\Big) \;.
%		&\leq
%			\Big(\abs{N}\log\frac{1}{\varepsilon}\Big) \abs{\partial^- N(J)}
	\end{equation}
\end{proposition}
\begin{proof} %\comment{Skip or move to the appendix?}
	Since $\varphi$ is strictly positive, $\mu^t=\mu\Phi^t$ and $\mu^t\widehat{\Phi}_J$  are fully supported for every $t>0$.
	Thus, for every $t>0$, we have
	\begin{align}
		\dot{D}_J(t) &=
			\sum_{w\in\Sigma^J} (\mu^t L)([w])\log\frac{\mu^t([w])}{\lambda([w])}
			&& \text{(Proposition~\ref{prop:IPS:entropy-change:in-terms-of-generator})} \\
		&=
			\abs{J}\sum_{w\in\Sigma^J} \big((\mu^t\widehat{\Phi}_J)([w])-\mu^t([w])\big)\log\frac{\mu^t([w])}{\lambda([w])}
			&& \text{(Observation~\ref{obs:IPS:generator-in-terms-of-asynchronous})} \\
		&=
			\abs{J}\sum_{w\in\Sigma^J}
			\mathrlap{
			\Bigg(
				(\mu^t\widehat{\Phi}_J)([w])
				\log\bigg(\frac{\mu^t([w])}{\lambda([w])}\cdot\frac{(\mu^t\widehat{\Phi}_J)([w])}{(\mu^t\widehat{\Phi}_J)([w])}\bigg)
				-
				\mu^t([w])\log\frac{\mu^t([w])}{\lambda([w])}
			\Bigg)
			} \\
		&=
			\abs{J}\Big(
				D_J\big(\mu^t\widehat{\Phi}_J \relto[\big] \lambda \big)
				-
				D_J\big(\mu^t\widehat{\Phi}_J \relto[\big] \mu\Phi^t \big)
				-
				D_J(\mu^t \relto \lambda)
			\Big)
	\end{align}
	as claimed.
\end{proof}

\subsection{Noise decomposition}
\label{sec:IPS:noise-decomposition}

Let $\Phi=(\Phi^t)_{t\geq 0}$ be a $d$-dimensional IPS with a strictly positive local transition rule $\varphi\colon\Sigma^N\times\Sigma\to(0,1)$ and clock rate~$1$.  %Let $L$ denote the generator of~$\Phi$.
Let $\lambda_q=\otimes_{i\in\ZZ^d}q$ be a Bernoulli measure with marginal~$q$, and suppose that $\lambda_q$ is stationary under~$\Phi$.

As in Section~\ref{sec:PCA:ergodicity}, we start by representing $\Phi$ as a perturbation of another IPS with noise.  Namely, let $\kappa>0$ be such that $\varphi(a_N,b)\geq\kappa q(b)$ for every $a_N\in\Sigma^N$ and $b\in\Sigma$, and decompose $\varphi$ as
\begin{equation}
	\varphi(a_N,b) = \kappa q(b) + (1-\kappa)\psi(a_N,b) \;,
\end{equation}
where $\psi\colon\Sigma\times\Sigma\to[0,1]$ is another local transition rule.
In order to draw a sample from $\varphi(a_N,\cdot)$, we can flip a biased coin with parameter~$\kappa$ to decide whether to sample from $q$ or from $\psi(a_N,\cdot)$.
In light of the colouring theorem of Poisson processes (see \ac{eg},~\cite[Section~5.1]{Kingman1993}), we can thus interpret the evolution of~$\Phi$ as follows:  Each site $i\in\ZZ$ is assigned two independent Poisson clocks $\xi^*_i$ and $\xi^\circ_i$ with rates $\kappa$ and $(1-\kappa)$ respectively.  The clocks attached to different sites are independent.  At each tick of $\xi^*_i$, the symbol at site $i$ is resampled according to~$q$, and at each tick of $\xi^\circ_i$, the symbol at site~$i$ is resampled according to~$\psi$, depending on its current neighbourhood pattern.

%Let
%\begin{align}
%	(\widetilde{L}f)(x) &\isdef
%		\sum_{k\in\ZZ^d} \sum_{b\in\Sigma}\psi\big((\sigma^k x)_N, b\big)\big(f(x^{(k\to b)})-f(x)\big) \;, \\
%	(L^*f)(x) &\isdef
%		\sum_{k\in\ZZ^d} \sum_{b\in\Sigma}q(b)\big(f(x^{(k\to b)})-f(x)\big)
%\end{align}
%be the generators for the IPS with local rules $\psi$ and $q$.
%\begin{observation}[Noise decomposition]
%\label{obs:IPS:noise-decomposition}
%	The generator of $\Phi$ can be written as
%	\begin{equation}
%		L = \kappa L^* + (1-\kappa)\widetilde{L} \;.
%	\end{equation}
%\end{observation}

The generator of $\Phi$ can be decomposed as
\begin{equation}
	L = \kappa L^* + (1-\kappa)L^\circ \;,
\end{equation}
where
\begin{align}
	(L^\circ f)(x) &\isdef
		\sum_{k\in\ZZ^d} \sum_{b\in\Sigma}\psi\big((\sigma^k x)_N, b\big)\big(f(x^{(k\to b)})-f(x)\big) \;, \\
	(L^* f)(x) &\isdef
		\sum_{k\in\ZZ^d} \sum_{b\in\Sigma}q(b)\big(f(x^{(k\to b)})-f(x)\big)
\end{align}
are the generators for the IPS with local rules $\psi$ and $q$.
We interpret $L^*$ as the generator of the \emph{asynchronous, zero-range, memoryless noise} with \emph{error rate} $1$ and \emph{replacement distribution} $q$.

Clearly, $\lambda_q$ is preserved by noise, that is, $\lambda_q L^*=0$.  It follows that:
\begin{lemma}
\label{lem:IPS:decomposition:invariance}
	If $\Phi$ preserves the Bernoulli measure~$\lambda_q$ and $\kappa<1$, then $\lambda_q L^\circ=0$.
\end{lemma}
\begin{proof}
	We have $\lambda_q L = \kappa\lambda_q L^* + (1-\kappa)\lambda_q L^\circ$.
	Since $\lambda_q L=0=\lambda_q L^*$, we must also have $\lambda_q L^\circ=0$.
\end{proof}

Let $(\widehat{\Phi}_J)_{J\Subset\ZZ^d}$, $(\widehat{\Psi}_J)_{J\Subset\ZZ^d}$ and $(\Theta_J)_{J\Subset\ZZ^d}$ be, respectively, the families of asynchronous updating kernels associated with $\varphi$, $\psi$ and $q$, defined as in~\eqref{eq:asynchronous-kernel:finite-set}.
Note that $(\Theta_J)_{J\Subset\ZZ^d}$ corresponds to the (asynchronous, zero-range, memoryless) noise with error rate~$1$ and replacement distribution~$q$.
%We call $(\Theta_J)_{J\Subset\ZZ^d}$ the \emph{asynchronous, zero-range, memoryless noise} with \emph{error probability} $\kappa$ and \emph{replacement distribution} $q$.

\begin{observation}[Noise decomposition]
\label{obs:IPS:noise-decomposition}
	$\widehat{\Phi}_J=\kappa\Theta_J + (1-\kappa)\widehat{\Psi}_J$ for every $J\Subset\ZZ^d$.
\end{observation}

\subsection{Almost depletion of entropy}
\label{sec:IPS:entropy-almost-depletion}

Thanks to Proposition~\ref{prop:IPS:entropy-change:in-terms-of-asynchronous}, analysing the evolution of entropy in positive-rate IPS boils down to examining the change in entropy under asynchronous updating kernels.
As in the case of a PCA, asynchronous updating of a set $J$ using any local rule~$\psi$ diffuses the entropy only through the boundary of~$J$ while asynchronous noise dampens the entropy across the entire set~$J$.  In absence of a continuous-time analogue of inequality~\eqref{eq:PCA:entropy-diffusion:1}, however, we need a somewhat more elaborate argument to conclude exponential ergodicity.  

\begin{lemma}[Local diffusion of entropy] %[Effect of a PCA on entropy]
\label{lem:IPS:entropy-diffusion}
	%Let $\widehat{\Psi}$ be a $d$-dimensional PCA with local transition rule $\varphi\colon\Sigma^N\times\Sigma\to[0,1]$,
	Let $\widehat{\Psi}=(\widehat{\Psi}_J)_{J\Subset\ZZ^d}$ be the family of asynchronous updating kernels associated with a local transition rule $\psi\colon\Sigma^N\times\Sigma\to[0,1]$, and suppose that $\widehat{\Psi}_J$ preserves a Bernoulli measure~$\lambda=\lambda_q$.
	Then,
	\begin{equation}
		D_J(\mu\widehat{\Psi}_J\relto\lambda) \leq
			D_J(\mu\relto\lambda) + \abs{N}\log(\sfrac{1}{q_\smallest})\cdot\frac{\abs{\partial^- N(J)}}{\abs{J}}
	\end{equation}
	for every probability measure~$\mu$ and every finite set $J\subseteq\ZZ^d$,
	where $q_\smallest\isdef\min\{q(a): a\in\Sigma\}$ and
	$\partial^-N(J)\isdef\{k\in J: N(k)\not\subseteq J\}$.
\end{lemma}
\begin{proof}
	By convexity, we have
	\begin{equation}
		D_J(\mu\widehat{\Psi}_J\relto\lambda) =
			D_J\bigg(\frac{1}{\abs{J}}\sum_{k\in J}\mu\widehat{\Psi}_k\relto[\bigg]\lambda\bigg)
		\leq
			\frac{1}{\abs{J}}\sum_{k\in J}D_J(\mu\widehat{\Psi}_k\relto\lambda) \,.
	\end{equation}
	Using the weak data processing inequality~(Proposition~\ref{prop:WDPI}), for every $k\in J$, we have
%	\begin{equation}
%		D_J(\mu\widehat{\Psi}_k\relto\lambda)
%			\leq
%			\begin{cases}
%				D_J(\mu\relto\lambda) &\text{if $N(k)\subseteq J$,} \\
%				D_{J\cup N(k)}(\mu\relto\lambda) &\text{if $N(k)\not\subseteq J$.}
%			\end{cases}
%	\end{equation}
	\begin{align}
		D_J(\mu\widehat{\Psi}_k\relto\lambda) &\leq
			\mathrlap{D_J(\mu\relto\lambda)}\qquad\qquad\qquad\qquad\qquad
			\text{if $N(k)\subseteq J$,} \\
		D_J(\mu\widehat{\Psi}_k\relto\lambda) &\leq
			\mathrlap{D_{J\cup N(k)}(\mu\relto\lambda)}\qquad\qquad\qquad\qquad\qquad
			\text{if $N(k)\not\subseteq J$.}
	\end{align}
	Hence,
	\begin{align}
		D_J(\mu\widehat{\Psi}_J\relto\lambda) &\leq
			\frac{1}{\abs{J}}\Bigg(
				\sum_{\substack{k\in J\\ N(k)\subseteq J}}D_J(\mu\relto\lambda) + 
				\sum_{\substack{k\in J\\ N(k)\not\subseteq J}} D_{J\cup N(k)}(\mu\relto\lambda)
			\Bigg) \\
		&=
			\frac{1}{\abs{J}}\bigg(
				\sum_{k\in J}D_J(\mu\relto\lambda) + 
				\sum_{k\in \partial^-N(J)}
					\Big(D_{J\cup N(k)}(\mu\relto\lambda)-D_J(\mu\relto\lambda)\Big)
			\bigg) \\
		&=
			D_J(\mu\relto\lambda) +
				\frac{1}{\abs{J}}\sum_{k\in \partial^-N(J)}
					\Big(D_{J\cup N(k)}(\mu\relto\lambda)-D_J(\mu\relto\lambda)\Big) \, .
			\label{eq:IPS:entropy-diffusion:proof:1}
	\end{align}
	Now, consider a site $k\in J$ such that $N(k)\not\subseteq J$.
	Letting $X$ and $Z$ be random configurations, respectively distributed according to $\mu$ and $\lambda$, we can write:
	\begin{align}
		D_{J\cup N(k)}(\mu\relto\lambda) - D_J(\mu\relto\lambda) &=
			D\big(X_{J\cup N(k)}\relto[\big]Z_{J\cup N(k)}\big) - 
			D(X_J\relto Z_J) \\
		&=
			D\big((X_{N(k)\setminus J}\given X_J)\relto[\big](Z_{N(k)\setminus J}\given Z_J)\big)
			&& \text{(chain rule)} \\
		&\leq
			\abs{N}\log(\sfrac{1}{q_\smallest}) \;.
			&& \text{(Lemma~\ref{lem:relative-entropy:bound})}
			\label{eq:IPS:entropy-diffusion:proof:2}
	\end{align}
	Combining~\eqref{eq:IPS:entropy-diffusion:proof:1} and~\eqref{eq:IPS:entropy-diffusion:proof:2} gives
	\begin{equation}
		D_J(\mu\widehat{\Psi}_J\relto\lambda) \leq
			D_J(\mu\relto\lambda) +
				\frac{\abs{\partial^-N(J)}}{\abs{J}}\cdot
					\abs{N}\log(\sfrac{1}{q_\smallest}) \;,
	\end{equation}
	as claimed.
\end{proof}

\begin{lemma}[Entropy decay]  %[Effect of noise on entropy]
\label{lem:IPS:noise:contraction}
	Let $(\Theta_J)_{J\Subset\ZZ^d}$ be the family of asynchronous updating kernels associated with memoryless noise
	with error rate~$1$ and replacement distribution~$q$,
	and let $\lambda=\lambda_q$ be the Bernoulli measure with marginal~$q$.
	Then,
	\begin{equation}
		D_J(\mu\Theta_J\relto\lambda)
		\leq
			\Big(1-\frac{1}{\abs{J}}\Big) D_J(\mu\relto\lambda)
	\end{equation}
	for every probability measure~$\mu$ and every finite set $J\subseteq\ZZ^d$.
\end{lemma}
\begin{proof}
	This is an application of the strong data processing inequality for non-interacting Markov chains with asynchronous updating (Proposition~\ref{prop:SDPI:asynchronous}).  Note that for $\theta(a,b)\isdef q(b)$, we can choose $\kappa=1$.
\end{proof}

Putting Proposition~\ref{prop:IPS:entropy-change:in-terms-of-asynchronous}, Observation~\ref{obs:IPS:noise-decomposition}, Lemmas~\ref{lem:IPS:entropy-diffusion}, and Lemma~\ref{lem:IPS:noise:contraction} together gives the following.
\begin{proposition}[Local entropy diffusion with decay] %[Effect of a positive-rate PCA on entropy]
\label{prop:IPS:relative-entropy:derivative}
	Let $\Phi$ be an IPS with a strictly positive local transition rule $\varphi\colon\Sigma^N\times\Sigma\to(0,1)$.
	%where $N\ni 0$.
	Suppose that $\Phi$ admits a stationary Bernoulli measure $\lambda=\lambda_q$.
	Let $\mu$ be a probability measure and $J\Subset\ZZ^d$.
	Let $\mu^t\isdef\mu\Phi^t$ and $D_J(t)\isdef D_J(\mu^t\relto\lambda)$.
	Then,
	\begin{equation}
		\dot{D}_J(t) \leq
			-\kappa D_{J}(t) + (1-\kappa)\abs{N}\log(\sfrac{1}{q_\smallest})\cdot\abs{\partial^- N(J)} \;,
	\end{equation}
	%for every probability measure~$\mu$ and every finite set $J\subseteq\ZZ^d$,
	for every $t>0$,
	where
	$\kappa \isdef \min\big\{\varphi(a_N,b)/q(b): a_N\in\Sigma^N,b\in\Sigma\big\}$ and
	$q_\smallest\isdef\min\{q(a): a\in\Sigma\}$.
\end{proposition}
\begin{proof}
	Consider the noisy representation
	\begin{equation}
		\varphi(a_N,b) = \kappa q(b) + (1-\kappa)\psi(a_N,b)
	\end{equation}
	as in Section~\ref{sec:IPS:noise-decomposition}.  As before, let $(\widehat{\Psi}_J)_{J\Subset\ZZ^d}$ and $(\Theta_J)_{J\Subset\ZZ^d}$ be, respectively, the families of asynchronous updating kernels associated with $\psi$ and $q$.
	We have
	\begin{align}
		D_J\big(\mu^t\widehat{\Phi}_J \relto[\big] \lambda \big) &=
			D_J\Big(\mu^t\big(\kappa\Theta_J + (1-\kappa)\widehat{\Psi}_J\big) \relto[\Big] \lambda \Big)
			&& \text{(Observation~\ref{obs:IPS:noise-decomposition})} \\
		&\leq
			\kappa D_J\big(\mu^t\Theta_J \relto[\big] \lambda \big)
			+
			(1-\kappa) D_J\big(\mu^t\widehat{\Psi}_J \relto[\big] \lambda \big)
			&& \text{(convexity)} \\
		&\leq
			\begin{multlined}[t]
			\kappa\Big(1-\frac{1}{\abs{J}}\Big) D_J(\mu^t\relto\lambda)
			\\ +
			(1-\kappa)\bigg(
				D_J(\mu^t\relto\lambda) + \abs{N}\log(\sfrac{1}{q_\smallest})\cdot\frac{\abs{\partial^- N(J)}}{\abs{J}}
			\bigg) \;.
			\end{multlined}
			&& \text{(Lemmas~\ref{lem:IPS:entropy-diffusion} and~\ref{lem:IPS:noise:contraction})}
	\end{align}
	Therefore,
	\begin{align}
		\dot{D}_J(t) &=
			\abs{J}\Big(
				D_J\big(\mu^t\widehat{\Phi}_J \relto[\big] \lambda \big)
				-
				D_J\big(\mu^t\widehat{\Phi}_J \relto[\big] \mu^t \big)
				-
				D_J(\mu^t \relto \lambda)
			\Big)
			&& \text{(Proposition~\ref{prop:IPS:entropy-change:in-terms-of-asynchronous})} \\
		&\leq
			\abs{J}\Big(
				D_J\big(\mu^t\widehat{\Phi}_J \relto[\big] \lambda \big)
				-
				D_J(\mu^t \relto \lambda)
			\Big)
			&& \text{($D_J\big(\mu^t\widehat{\Phi}_J \relto[\big] \mu^t \big)\geq 0$)} \\
		&\leq
			-\kappa D_{J}(\mu\relto\lambda) + (1-\kappa)\abs{N}\log(\sfrac{1}{q_\smallest})\cdot\abs{\partial^- N(J)} \;,
	\end{align}
	for every $t>0$.
	This proves the proposition.
\end{proof}

The latter proposition implies that, except for a boundary term, the relative entropy of a finite region exponential decays.
\begin{proposition}[Evolution of entropy]
\label{prop:IPS:relative-entropy:evolution}
	Let $\Phi$ be an IPS with a strictly positive local transition rule $\varphi\colon\Sigma^N\times\Sigma\to(0,1)$, and suppose that $\Phi$ admits a stationary Bernoulli measure $\lambda=\lambda_q$.
%	Let $\mu$ be a probability measure and $J\Subset\ZZ^d$.
%	Let $\mu^t\isdef\mu\Phi^t$ and $D_J(t)\isdef D_J(\mu^t\relto\lambda)$.
	Then,
	\begin{equation}
		D_J(\mu\Phi^t\relto\lambda) \leq
			D_J(\mu\Phi^t\relto\lambda)\ee^{-\kappa t}
			+
			\frac{1-\kappa}{\kappa}\abs{N}\log(\sfrac{1}{q_\smallest})\cdot\abs{\partial^- N(J)}
%		&\leq
%			\log(\sfrac{1}{q_\smallest})\cdot\abs{J}\cdot\ee^{-\kappa t}
%			+
%			\frac{1-\kappa}{\kappa}\abs{N}\log(\sfrac{1}{q_\smallest})\cdot\abs{\partial^- N(J)} 
	\end{equation}
	for every probability measure~$\mu$, every finite set $J\subseteq\ZZ^d$, and every $t\geq 0$, where
	$\kappa \isdef \min\big\{\varphi(a_N,b)/q(b): a_N\in\Sigma^N,b\in\Sigma\big\}$
	and
	$q_\smallest\isdef\min\{q(a): a\in\Sigma\}$.
\end{proposition}
\begin{proof} %\comment{Skip or move to an appendix}
	Let $\mu$ be a probability measure and $J\Subset\ZZ^d$.
	Let $\mu^t\isdef\mu\Phi^t$ and $D_J(t)\isdef D_J(\mu^t\relto\lambda)$.
	According to Proposition~\ref{prop:IPS:relative-entropy:derivative},
	\begin{equation}
		\dot{D}_J(t) + \kappa D_{J}(t) \leq
			 (1-\kappa)\abs{N}\log(\sfrac{1}{q_\smallest})\cdot\abs{\partial^- N(J)}
	\end{equation}
	for $t>0$.
	Multiplying by $\ee^{\kappa t}$ and integrating with respect to~$t$ yields the result.
\end{proof}

\subsection{Propagation of influence}
\label{sec:IPS:influence-propagation}

Let $A\Subset\ZZ^d$ be a finite set of sites and $t\geq 0$.  In a PCA, the state $X^t_A$ of $A$ at time~$t$ can depend on the initial configuration $X^0$ only through its restriction~$X^0_{N^t(A)}$ to $N^t(A)$, where $N$ denotes the dependence neighbourhood of the local rule.  In other words, there is no communication from outside $N^t(A)$ to $A$ in $t$ steps.  By contrast, in an IPS, the set of sites whose states at time~$0$ may influence the state of~$A$ at time~$t$ is random and unbounded, due to its dependence on the family of Poisson clocks~$(\xi_k)_{k\in\ZZ^d}$.  Let us denote this random set by~$\Xi^t(A)$.  The purpose of this section is to provide a simple concentration inequality for~$\Xi^t(A)$, showing that, for $\ell>1$, the probability that $\Xi^t(A)$ exceeds $N^{\ell t}(A)$ is exponentially small in~$t\ell\log\ell$.

The set $\Xi^t(A)$ is recursively defined as follows.
\begin{itemize}
    \item If none of the Poisson clocks $\xi_k$ (for $k\in A$) ticks during the time interval $[0,t]$, we let $\Xi^t(A)\isdef A$.
    \item Otherwise, we let $s$ denote the time of the last tick of the Poisson clocks in $A$ during $[0,t]$ and recursively set $\Xi^t(A)\isdef\Xi^{s-}\big(A\cup N(k)\big)$, where $k$ is the index of the clock that has ticked at time~$s$.
\end{itemize}
It is standard to show that this recursion ends almost surely after finitely many steps and thus, $\Xi^t(A)$ is well defined.
\begin{observation}[Influence region]
\label{obs:influence:markov}
    Let $(X^t)_{t\geq 0}$ be a random trajectory of an IPS with dependence neighbourhood~$N$.
    Then, for every $A\Subset\ZZ^d$ and $t\geq 0$, we have
    \begin{equation}
        \PP\big(X^t_A\in\cdot\given X^0, (\xi_k)_{k\in\ZZ^d}\big)=\PP\big(X^t_A\in\cdot\given[\big] X^0_{\Xi^t(A)}, (\xi_k)_{k\in\ZZ^d}\big)
    \end{equation}
    almost surely.
\end{observation}

The construction of $\Xi^t(A)$ can be viewed as a growth process going backwards in time.
The growth model is defined as follows.  There is again a family $(\xi'_k)_{k\in\ZZ^d}$ of independent Poisson clocks with rate~$1$.  At the beginning, a set $A$ of sites are \emph{infected}.  Once the Poisson clock at an infected site $k$ ticks, the entire neighbourhood $N(k)$ of that sites becomes infected.  Thus, if $\Pi^t(A)$ denotes the set of sites that are infected at time~$t$, we have $\Pi^s(A)=\Pi^{s-}(A)\cup N(k)$ when $k\in\Pi^{s-}(A)$ and $s\in\xi'_k$.
It is clear that if we couple the two processes by letting $\xi'_k=t-\xi_k$ for every $k\in\ZZ^d$, then $\Pi^t(A)=\Xi^t(A)$.

Let us extend the notation $N^s(A)$ by allowing $s$ to be a non-negative real number, in which case $N^s(A)\isdef N^{\floor{s}}(A)$.
\begin{lemma}[Concentration of influence/infected region]
\label{lem:influence-growth}
    %Let $\Phi$ be an IPS on $\Sigma^{\ZZ^d}$ with rule neighbourhood $0\in N\Subset\ZZ^d$.
    Let $0\in N\Subset\ZZ^d$, and let $\Pi^t$ be the growth process with neighbourhood~$N$ as described above.
    For every $A\Subset\ZZ^d$, $t\geq 0$, and $\ell>1$, we have
    \begin{equation}
        \PP\big(\Pi^t(A)\not\subseteq N^{\ell t}(A)\big) \leq
            \abs{A}\ee^{-(\ell\log\sfrac{\ell}{\rho} - \ell + 1)t} \;,
            %\abs{A}\ee^{-(\ell\log\ell - (1+\log\rho)\ell + 1)t} \;,
    \end{equation}
    where $\rho\isdef\abs{N}$.
    In particular, for every $t\geq 0$ and $\varepsilon>0$, we have
    \begin{equation}
        \PP\big(\Pi^t(A)\not\subseteq N_{\varepsilon,t}(A)\big) \leq \varepsilon \;,
    \end{equation}
    where $N_{\varepsilon,t}(A)\isdef N^{s(\varepsilon,t,\abs{A})}(A)$ with $s(\varepsilon,t,a)\isdef \max\big\{8\rho t, \log\sfrac{a}{\varepsilon}\big\}$.
\end{lemma}
\begin{proof}
    Consider the directed graph with vertex set $\ZZ^d$ in which there is a edge from $i$ to $j$ if $j\in N(i)\setminus\{i\}$.
    In order to have $\Pi^t(A)\not\subseteq N^{\ell t}(A)$, there must exist an infection path from $A$ to $\ZZ^d\setminus N^{\ell t}(A)$ in the time interval~$[0,t]$, that is, a path $u_0\to u_1\to\cdots\to u_n$ for which there exists a sequence $s_1,s_2,\ldots,s_n\in\RR$ with
    \begin{itemize}
        \item $0<s_1<s_2<\cdots<s_n\leq t$,
        \item $s_i\in\xi'_{u_{i-1}}$ for $i=1,2,\ldots,n$.
    \end{itemize}

    For $n\in\NN$, let $Q_n$ denote the set of all paths $\underline{u}=u_0\to u_1\to\cdots\to u_n$ from $A$ to $\ZZ^d\setminus N^{\ell t}(A)$.
    Note that any such path must be longer than $\ell t$.  %, that is, $n>\ell t$.
    Given such a path, let
    \begin{align}
        S_0 &\isdef 0 \;, \\
        S_i &\isdef \inf\big\{s: \text{$s>S_{i-1}$ and $s\in\xi'_{u_{i-1}}$}\big\} \;,  \qquad\text{for $i=1,2,\ldots,n$.}
    \end{align}
    In order for $\underline{u}$ to be an infection path, we must have $S_n<t$.  The inter-arrival times $S_i-S_{i-1}$ (for $i=1,2,\ldots,n$) are \ac{iid} exponential random variables with rate~$1$, hence
    \begin{equation}
        \PP(S_n<t) = \PP(\dPois(t)\geq n) \;.
    \end{equation}

    By the union bound, we have
    \begin{align}
        \PP\big(\Pi^t(A)\not\subseteq N^{\ell t}(A)\big)
        &\leq
            \sum_{n>\ell t}\sum_{\underline{u}\in Q_n}\PP(\dPois(t)\geq n) \\
        &=
            \sum_{n>\ell t}\abs{Q_n}\sum_{m\geq n}\PP(\dPois(t)=m) \\
        &=
            \sum_{m>\ell t}\PP(\dPois(t)=m)\sum_{~~~~\mathclap{n: \ell t<n\leq m}~~~~}\abs{Q_n} \\
        &\leq
            \sum_{m>\ell t}\ee^{-t}\frac{t^m}{m!}\times\abs{A}\rho^m \\
        &=
            \abs{A}\ee^{(\rho-1)t}\sum_{m>\ell t}\ee^{-\rho t}\frac{(\rho t)^m}{m!} \\
        &=
            \abs{A}\ee^{(\rho-1)t}\,\PP\big(\dPois(\rho t)>\ell t\big) \;.
    \end{align}
    For a Poisson distribution, the Chernoff bound gives
    \begin{equation}
        \PP(\dPois(\mu)\geq a) \leq
            \ee^{-a\log\sfrac{a}{\mu}+a-\mu} \;, \qquad\text{when $a>\mu$.}
    \end{equation}
    It follows that
    \begin{equation}
        \PP\big(\Pi^t(A)\not\subseteq N^{\ell t}(A)\big) \leq
            \abs{A}\ee^{-(\ell\log\sfrac{\ell}{\rho} - \ell + 1)t} \;,
    \end{equation}
    as claimed.

    To verify the second claim, let $\ell\isdef s(\varepsilon,t,\abs{A})/t$.
    From $s(\varepsilon,t,\abs{A})\geq 8\rho t$, we get
    \begin{equation}
        \ell\log\sfrac{\ell}{\rho} - \ell + 1 \geq \ell(\log 8 -1)+1 \geq \ell
    \end{equation}
    From $s(\varepsilon,t,\abs{A})\geq\log\sfrac{\abs{A}}{\varepsilon}$, it follows
    \begin{equation}
        \abs{A}\ee^{-(\ell\log\sfrac{\ell}{\rho} - \ell + 1)t} \leq
            \abs{A}\ee^{-\ell t} \leq \varepsilon \;,
    \end{equation}
    as claimed.
\end{proof}

\subsection{Bootstrapping} %{Continuous-time bootstrap lemma}
\label{sec:bootstrap:continuous-time}

% According to Proposition~\ref{prop:IPS:relative-entropy:almost-depletion},
% in the setting of Theorem~\ref{thm:main:IPS:ergodicity}, we have
% \begin{equation}
% 	D_J(\mu\Phi^t\relto\lambda) \leq \alpha\abs{\partial^- N(J)}
% 	\qquad\textup{for all $t\geq\beta\log\frac{\abs{J}}{\abs{\partial^- N(J)}}$,}
% \end{equation}
% uniformly in~$\mu$, for some constants $\alpha,\beta>0$.  In this section, we show that the above condition is enough to ensure $D_J(\mu\Phi^t\relto\lambda)$ exponentially decays to~$0$ as $t\to\infty$, as claimed in Theorem~\ref{thm:main:IPS:ergodicity}.
% \comment{This paragraph mentions Proposition~\ref{prop:IPS:relative-entropy:almost-depletion}, but the proof below directly uses Proposition~\ref{prop:IPS:relative-entropy:evolution}.}

Suppose $\Phi$ is a positive-rate IPS admitting a stationary Bernoulli measure~$\lambda$.
According to Proposition~\ref{prop:IPS:relative-entropy:evolution}, the relative entropy $D_J(\mu\Phi^t\relto\lambda)$ of a finite region $J\Subset\ZZ^d$ with respect to $\lambda$ decays exponentially up to a residual term of order the size of the boundary of~$J$.  In this section, we use a bootstrapping argument to show that this condition in fact implies that $D_J(\mu\Phi^t\relto\lambda)$ exponentially decays to~$0$ as $t\to\infty$.  This establishes Theorem~\ref{thm:main:IPS:ergodicity}.

Given $A,B\Subset\ZZ^d$, we denote by $\ell(B:A)$ the maximum number of disjoint translations of $A$ that can be packed inside~$B$.

% \begin{lemma}[Bootstrap lemma]
% \label{lem:bootstrap:IPS}
% 	Let $\Phi$ be a IPS on $\Sigma^{\ZZ^d}$ with dependence neighbourhood $0\in N\Subset\ZZ^d$ and suppose that $\Phi$ admits a stationary Bernoulli measure $\lambda$.
% 	Let $\tau,\delta\colon\FinSubset{\ZZ^d}\to[0,\infty)$ be two functions such that, for every $\mu\in\family{P}(\Sigma^{\ZZ^d})$ and $A\Subset\ZZ^d$, we have
% 	\begin{itemize}
% 		\item $D_A(\mu\Phi^t\relto\lambda)\leq\delta(A)$ for every $t\geq\tau(A)$.
% 	\end{itemize}
% 	Then, for every $\varepsilon>0$, $\mu\in\family{P}(\Sigma^{\ZZ^d})$ and $A,B\Subset\ZZ^d$, we have
% 	\begin{itemize}
% 		\item $D_A(\mu\Phi^t\relto\lambda)\leq\dfrac{\delta(B)}{\ell\big(B:N_{\varepsilon,t}(A)\big)} + H(\varepsilon) + \varepsilon\abs{A}\log(\sfrac{1}{q_\smallest})$ for every $t\geq\tau(B)$.
% 	\end{itemize}
% \end{lemma}
\begin{lemma}[Bootstrap lemma]
\label{lem:bootstrap:IPS:v2}
	Let $\Phi$ be an IPS on $\Sigma^{\ZZ^d}$ with dependence neighbourhood $0\in N\Subset\ZZ^d$ and suppose that $\Phi$ admits a stationary Bernoulli measure $\lambda=\lambda_q$.
	Then, for every $0<\varepsilon<\sfrac{1}{2}$, $\mu\in\family{P}(\Sigma^{\ZZ^d})$, $A,B\Subset\ZZ^d$, and $t\geq 0$, we have
	\begin{equation}
		D_A(\mu\Phi^t\relto\lambda)\leq\dfrac{\sup_\nu D_B(\nu\Phi^t\relto\lambda)}{\ell\big(B:N_{\varepsilon,t}(A)\big)} + H(\varepsilon) + \varepsilon\abs{A}\log(\sfrac{1}{q_\smallest}) \;,
	\end{equation}
    where $N_{\varepsilon,t}(A)$ is as in Lemma~\ref{lem:influence-growth} and $q_\smallest\isdef\min\{q(a): a\in\Sigma\}$.
\end{lemma}
\begin{proof}
	Let $\ell\isdef\ell\big(B:N_{\varepsilon,t}(A)\big)$.
	Let $k_1,k_2,\ldots,k_\ell\in\ZZ^d$ be such that the sets $A_i\isdef k_i+N_{\varepsilon,t}(A)$ are disjoint and contained in~$B$.
	Let $(\xi_k)_{k\in\ZZ^d}$ be a family of independent Poisson clocks with rate~$1$, which we will use to construct the trajectory of $\Phi$.
	For $i=1,2,\ldots,\ell$, let $W_i$ be a Bernoulli random variable indicating whether the influence of $\ZZ^d\setminus N_{\varepsilon,t}(A_i)$ reaches $A_i$ by time~$t$, that is, whether $\Xi_t(A_i)\not\subseteq N_{\varepsilon,t}(A_i)$.
    %Since $W_i$ is determined by the Poisson clocks in $N_{\varepsilon,t}(A_i)$, the random variables
    Clearly, $W_1,W_2,\ldots,W_\ell$ are independent, with $\PP(W_i=\symb{1})\leq \varepsilon$.
	%Let $L\isdef W_1+W_2+\cdots+W_\ell$.
	
    Let $Z$ be a random configuration distributed according to~$\lambda$.
	Given~$\mu$, construct a random configuration $X$, independently of $(\xi_k)_{k\in\ZZ^d}$, by choosing the patterns $X_{N_{\varepsilon,t}(A_i)}$ (for $i=1,2,\ldots,\ell$) independently according to (the corresponding translations of) $\mu_{N_{\varepsilon,t}(A)}$.
    The rest of $X$ can be chosen arbitrarily.
	Construct a trajectory $(X^t)_{t\geq 0}$ of $\Phi$ with initial configuration $X^0=X$ and the Poisson clocks $(\xi_k)_{k\in\ZZ^d}$.
    % For $i=1,2,\ldots,\ell$, also construct a random configuration $Y^i\sim \sigma^{-k_i}\mu$ and a corresponding trajectory $(Y^{i,t})_{t\geq 0}$ with $Y^{i,0}=Y^i$ using the same Poisson clocks $(\xi_k)_{k\in\ZZ^d}$.
    % Note that $X_{N_{\varepsilon,t}(A_i)}$ and $Y^i_{N_{\varepsilon,t}(A_i)}$ have the same distribution.  As a consequence, given $W_i=\symb{0}$, $X^t_{A_i}$ and $Y^{i,t}_{A_i}$ have the same conditional distributions.
	
	Let $\vect{w}\in\{\symb{0},\symb{1}\}^\ell$.  Clearly, given $\vect{W}=\vect{w}$, the patterns $X^t_{A_i}$ with $i\in\{1,2,\ldots,\ell\}$ and $w_i=\symb{0}$ are independent.
	Furthermore, given $W_i=\symb{0}$, the pattern $X^t_{A_i}$ is independent of $(W_j)_{j\neq i}$.
	Therefore,
	\begin{align}
		%D_B(\mu\Phi^t\relto\lambda) &= D(X^t_B\relto Z_B) \\
		\MoveEqLeft
		D\big((X^t_B\given\vect{W}=\vect{w})\relto[\big]Z_B\big) \\
		&\geq
			D\Big(\big(X^t_{\bigcup_{i: w_i=\symb{0}} A_i}\given[\big]\vect{W}=\vect{w}\big)\relto[\Big]Z^t_{\bigcup_{i: w_i=\symb{0}} A_i}\Big)
			&& \text{(Proposition~\ref{prop:WDPI})} \\
		&=
			\sum_{i:w_i=\symb{0}} D\big((X^t_{A_i}\given\vect{W}=\vect{w})\relto[\big] Z_{A_i}\big)
			&& \text{(chain rule + independence)} \\
		&=
			\sum_{i=1}^\ell (1-w_i) D\big((X^t_{A_i}\given W_i=\symb{0})\relto[\big] Z_{A_i}\big)
			&& \text{(independence given $W_i=\symb{0}$)}
    \end{align}
	Averaging over $\vect{w}$, we get
    \begin{equation}
		D\big((X^t_B\given\vect{W})\relto[\big]Z_B\big)
		      \geq
			\sum_{i=1}^\ell\PP(W_i=\symb{0})D\big((X^t_{A_i}\given W_i=\symb{0})\relto[\big] Z_{A_i}\big)
    \end{equation}
    % To bound the $i$th term on the right-hand side, construct a random configuration $Y^i\sim \sigma^{-k_i}\mu$ and a corresponding trajectory $(Y^{i,t})_{t\geq 0}$ with $Y^{i,0}=Y^i$ using the same Poisson clocks $(\xi_k)_{k\in\ZZ^d}$.
    % Note that $X_{N_{\varepsilon,t}(A_i)}$ and $Y^i_{N_{\varepsilon,t}(A_i)}$ have the same distribution.  As a consequence, given $W_i=\symb{0}$, $X^t_{A_i}$ and $Y^{i,t}_{A_i}$ have the same conditional distributions.
    To bound the $i$th term on the right-hand side,
    % construct a random configuration $Y\sim \sigma^{-k_i}\mu$ and a corresponding trajectory $(Y^t)_{t\geq 0}$ with $Y^0=Y$
    construct a random trajectory $(Y^t)_{t\geq 0}$ with $Y^0\sim\sigma^{-k_i}\mu$
    using the same Poisson clocks $(\xi_k)_{k\in\ZZ^d}$.
    Note that $X^0_{N_{\varepsilon,t}(A_i)}$ and $Y^0_{N_{\varepsilon,t}(A_i)}$ have the same distribution.  As a consequence, given $W_i=\symb{0}$, $X^t_{A_i}$ and $Y^t_{A_i}$ have the same conditional distributions.
    Thus,   
	\begin{align}
        \MoveEqLeft
		\PP(W_i=\symb{0})D\big((X^t_{A_i}\given W_i=\symb{0})\relto[\big] Z_{A_i}\big) \\
        &=
			\PP(W_i=\symb{0})D\big((Y^t_{A_i}\given W_i=\symb{0})\relto[\big] Z_{A_i}\big) \\
		&=
            \mathrlap{%
			D\big((Y^t_{A_i}\given W_i)\relto[\big] Z_{A_i}\big)
				- \PP(W_i=\symb{1})D\big((Y^t_{A_i}\given W_i=\symb{1})\relto[\big] Z_{A_i}\big)
            } \\
		&\geq
			D(Y^t_{A_i}\relto Z_{A_i})
			-
			\varepsilon\abs{A}\log(\sfrac{1}{q_\smallest})
			&& \text{(Lemmas~\ref{lem:relative-entropy:conditioning-on-common-info} and~\ref{lem:relative-entropy:bound})} \\
		&=
		      D_A(\mu\Phi^t\relto\lambda) - \varepsilon\abs{A}\log(\sfrac{1}{q_\smallest}) \;.
			&& \text{(construction)}
	\end{align}
    It follows that
    \begin{equation}
		D\big((X^t_B\given\vect{W})\relto[\big]Z_B\big)
        \geq
            \ell D_A(\mu\Phi^t\relto\lambda) - \ell \varepsilon\abs{A}\log(\sfrac{1}{q_\smallest}) \;.
    \end{equation}
    
	Using Observation~\ref{obs:csiszar-identity}, we have
	\begin{align}
		D\big((X^t_B\given\vect{W})\relto[\big]Z_B\big) &=
			D(X^t_B\relto Z_B) + I(X^t_B : \vect{W}) \\
		&\leq
			D(X^t_B\relto Z_B) + H(\vect{W}) \\
		&\leq
			D(X^t_B\relto Z_B) + \ell H(\varepsilon) \;.
	\end{align}
	Therefore,
	\begin{equation}
		D(X^t_B\relto Z_B) \geq
			\ell D_A(\mu\Phi^t\relto\lambda) - \ell \varepsilon\abs{A}\log(\sfrac{1}{q_\smallest}) - \ell H(\varepsilon) \;,
	\end{equation}
	or equivalently,
	\begin{equation}
		D_A(\mu\Phi^t\relto\lambda) \leq
			\frac{1}{\ell}D(X^t_B\relto Z_B) + H(\varepsilon) + \varepsilon\abs{A}\log(\sfrac{1}{q_\smallest}) \;.
	\end{equation}
	The result follows.
\end{proof}

%\subsection{Proof of Theorem~\ref{thm:main:IPS:ergodicity}}

We are now ready to prove the main theorem in continuous time.

\begin{proof}[Proof of Theorem~\ref{thm:main:IPS:ergodicity}]
    Let $r\in\NN$ be such that $N\subseteq\zinterval{-r,r}^d$.
	Let $A\Subset\ZZ^d$ be a finite set with diameter~$n$.  Without loss of generality, we can assume that $A=a+\zinterval{0,n-1}^d$ for some $a\in\ZZ^d$.
	Combining Proposition~\ref{prop:IPS:relative-entropy:evolution} with Lemma~\ref{lem:relative-entropy:bound}, for every $A\Subset\ZZ^d$, we have
	\begin{equation}
		D_A(\mu\Phi^t\relto\lambda) \leq
			\alpha_0\abs{A}\ee^{-\kappa t}
			+
			\beta_0\abs{\partial^- N(A)} 
        \label{eq:IPS:proof:decay-up-to-boundary}
	\end{equation}
	uniformly in $\mu$, where $\alpha_0\isdef\log(\sfrac{1}{q_\smallest})$ and $\beta_0\isdef\frac{1-\kappa}{\kappa}\abs{N}\log(\sfrac{1}{q_\smallest})$.
	In order to eliminate the boundary term~$\beta_0\abs{\partial^- N(A)}$, for every $t\geq 0$, we will choose a set $B_t\Subset\ZZ^d$ and an error margin $\varepsilon_t>0$ and apply the bootstrap lemma (Lemma~\ref{lem:bootstrap:IPS:v2}).

    Let $R_{\varepsilon,t,n}\isdef r s(\varepsilon,t,n^d)$, where $s(\varepsilon,t,n^d)=s(\varepsilon,t,\abs{A})$ is as in Lemma~\ref{lem:influence-growth}, and note that
    \begin{equation}
        N_{\varepsilon,t}(A)=N^{s(\varepsilon,t,\abs{A})}(A)
            \subseteq A+\zinterval{-R_{\varepsilon,t,n},R_{\varepsilon,t,n}}^d \;.
    \end{equation}
	We choose $B_t\isdef\zinterval{0,m_t(n+2R_{\varepsilon,t,n})-1}^d$ for some $m_t\in\NN$, which will be chosen later.  Observe that $\ell\big(B_t:N_{\varepsilon_t,t}(A)\big)=m_t^d$.  Furthermore, for every $k\in\NN$, we have
	$\abs[\big]{\partial^-N(\zinterval{0,k-1}^d)}\leq 2rdk^{d-1}$, hence
	\begin{equation}
		\abs{\partial^-N(B_t)} \leq
			2rd\,m_t^{d-1}(n+2R_{\varepsilon,t,n})^{d-1} \;.
	\end{equation}
	Therefore, for every $t\geq 0$,
	\begin{align}
		D_A(\mu\Phi^t\relto\lambda) &\leq
			\dfrac{\sup_\nu D_{B_t}(\nu\Phi^t\relto\lambda)}{\ell\big(B_t:N_{\varepsilon_t,t}(A)\big)}
				+ H(\varepsilon_t) + \varepsilon_t\abs{A}\log(\sfrac{1}{q_\smallest}) 
			&& \text{(Lemma~\ref{lem:bootstrap:IPS:v2})} \\
		&\leq
			\dfrac{\alpha_0\abs{B_t}\ee^{-\kappa t} + \beta_0\abs{\partial^- N(B_t)}}{\ell\big(B_t:N_{\varepsilon_t,t}(A)\big)}
			+ H(\varepsilon_t) + \varepsilon_t\abs{A}\log(\sfrac{1}{q_\smallest})
			&& \text{(by~\eqref{eq:IPS:proof:decay-up-to-boundary})} \\%\text{(Proposition~\ref{prop:IPS:relative-entropy:evolution})} \\
		&\leq
			%\begin{multlined}[t]
				\dfrac{
				\alpha_0 m_t^d(n+2R_{\varepsilon,t,n})^d\ee^{-\kappa t}
				+
				\beta_0\cdot2rd\,m_t^{d-1}(n+2R_{\varepsilon,t,n})^{d-1}
				}{m_t^d} %\\[1em]
				\mathrlap{%
				+ H(\varepsilon_t) + \varepsilon_t n^d\log(\sfrac{1}{q_\smallest}) 
				}
			%\end{multlined}
			\\
		&=
			\alpha_0 (n+2R_{\varepsilon,t,n})^d\ee^{-\kappa t}
			+
			\beta_0\cdot2rd\frac{(n+2R_{\varepsilon,t,n})^{d-1}}{m_t}
			+
			\mathrlap{%
                H(\varepsilon_t) + \varepsilon_t n^d\log(\sfrac{1}{q_\smallest}) 
			}
	\end{align}
	Now, let us choose $\varepsilon_t\isdef\ee^{-\kappa t}$ and $m_t\isdef\ee^{\kappa t}$, and note that, by Lemma~\ref{lem:influence-growth},
    \begin{equation}
        R_{\varepsilon,t,n} = r s(\varepsilon,t,n^d)
            = r\max\big\{8\rho t, \log(n^d/\varepsilon_t)\big\}
            \leq r(8\rho+\kappa)t + rd\log n \;.
    \end{equation}
	With this choice, for every $0<\beta_1<\kappa$, each term in the above upper bound becomes bounded by a constant multiple of $\ee^{-\beta_1 t}n^d$.
	Hence,
	\begin{equation}
		D_A(\mu\Phi^t\relto\lambda) \leq \alpha_1\ee^{-\beta_1 t} n^d
	\end{equation}
	for an appropriate constant~$\alpha_1>0$.
	Applying Pinsker's inequality (\ac{eg},~\cite[Lemma~11.6.1]{CT2006}) as in the discrete-time case yields
	\begin{equation}
		\norm{\mu\Phi^t-\lambda}_A \leq \alpha\ee^{-\beta t} n^{d/2} \;,
	\end{equation}
	where $\alpha\isdef\sqrt{\alpha_1/2}$ and $\beta\isdef\beta_1/2$.
\end{proof}

\section{Examples and characterizations}
\label{sec:characterizations}

\subsection{PCA with stationary Bernoulli measures}
\label{sec:characterizations:PCA}

The following is perhaps the simplest non-trivial example of a positive-rate PCA with a stationary Bernoulli measure.
\begin{example}[XOR + noise]
\label{exp:xor-with-noise}
    Let $0<\varepsilon<1$.
    Let $\Phi_\varepsilon$ be the one-dimensional PCA with alphabet $\Sigma\isdef\{\symb{0},\symb{1}\}$, dependence neighbourhood $N\isdef\{0,1\}$, and local rule %$\varphi\colon\Sigma^N\times\Sigma\to(0,1)$
    \begin{equation}
        \varphi_\varepsilon(w_0w_1,b) \isdef
            \begin{cases}
                1-\varepsilon   & \text{if $b=w_0+w_1\pmod{2}$,} \\
                \varepsilon     & \text{otherwise.}
            \end{cases}
    \end{equation}
    It is easy to verify that the uniform Bernoulli measure $\lambda_{\sfrac{1}{2}}$ is stationary under~$\Phi_\varepsilon$, hence $\Phi_\varepsilon$ is exponentially ergodic.
    In fact, it can be shown that the convergence under $\Phi_\varepsilon$ is super-exponential~\cite{Vaserstein1969} (see~\cite[Chapter~1]{TVS+1990}).
\end{example}

A class of PCA with stationary Bernoulli measures was introduced by Vasilyev, who also proved their exponential ergodicity under the positive-rate assumption~\cite{Vasilyev1978} (see also~\cite[Chapter~17]{TVS+1990}).
%Vasilyev introduced a family of PCA with stationary Bernoulli measures, for which he proved exponential ergodicity under the positive-rate assumption~\cite{Vasilyev1978} (see also~\cite[Chapter~17]{TVS+1990}).
Here, we present Vasilyev's construction in less generality to illustrate its idea.
\begin{example}[Vasilyev]
    Let $\Phi$ be a one-dimensional PCA with dependence neighbourhood $N=\{\ell,\ell+1,\ldots,r\}$ and a local transition rule $\varphi\colon\Sigma^N\times\Sigma\to[0,1]$ that satisfies the following property.
    Let $q\colon\Sigma\to(0,1)$ be a probability distribution on~$\Sigma$, and suppose that, for every $u=u_{\ell}u_{\ell+1}\cdots u_{r-1}\in\Sigma^{N\setminus\{r\}}$, the stochastic matrix $\varphi_u(a,b)\isdef\varphi(ua,b)$ has $q$ as a stationary distribution, that is, $q\varphi_u=q$.
    Observe that the PCA in Example~\ref{exp:xor-with-noise} is of this form.
    
    The stationarity of $\lambda_q$ can be shown using conditioning.  Vasilyev used an argument based on finite-state Markov chains to show that, if $\varphi$ is strictly positive, the PCA is ergodic with exponential convergence.

    To demonstrate the idea, let us assume that $r=0$, and first focus on how the distribution of the symbol at the origin evolves.
    Let $(X^t)_{t\geq 0}$ be a trajectory of $\Phi$.  Vasilyev observed that, given $(X^t_i: t\geq 0, i<0)$, the random variables $X^0_0, X^1_0, \ldots$ form a time-inhomogeneous Markov chain, in which the transition matrix at time $t$ is given by $\varphi_{X^t_{\zinterval{\ell,-1}}}$.  Since all these transition matrices have $q$ as the unique stationary distribution and their spectral gaps are bounded away from zero, the distribution of $X^t_0$ converges exponentially fast to~$q$.
    A similar argument shows that, for every finite block $\zinterval{i,j}$ of sites, the sequence $(X^t_{\zinterval{i,j}})_{t\geq 0}$ conditioned on $(X^t_k:t\geq 0, k<i)$ is a time-inhomogeneous Markov chain whose distribution converges exponentially fast to~$\bigotimes_{k\in\zinterval{i,j}}q$.
    Taking expectation with respect to $(X^t_k:t\geq 0, k<i)$ now establishes the exponential ergodicity of~$\Phi$.
    
    %The space-time random fields associated with Vasilyev's class of PCA was studied by Mairesse and Marcovici~\cite{MM2014}.
    %\comment{Not accurate:}\irene{That sounds accurate to me, if we also switch left and right in the condition.}
    Restricted to a binary alphabet $\Sigma\isdef\{\symb{0},\symb{1}\}$ and neighbourhood $N=\{-1,0\}$, every PCA with a stationary Bernoulli measure is of this type or its reflection~\cite{BGM1969,MM2014} (see also~\cite{MM2014a} and~\cite[Chapter 16]{TVS+1990}).
\end{example}

Vasilyev's class of PCA can be viewed as stochastic analogues of (deterministic) permutive CA, %~\cite{Hedlund1969}.
whereas the full class of PCA admitting stationary Bernoulli measures may be compared to the much larger class of surjective CA (see \ac{eg},~\cite{Kurka2003,Kari2026}). %~\cite[Chapter~5]{Kurka2003}.
%Another subclass of PCA with stationary Bernoulli measures arises as perturbations of surjective CA with permutation noise.

Another subclass of PCA admitting stationary Bernoulli measures arises from perturbing surjective CA with noise.
\begin{example}[Surjective CA + noise]
\label{exp:surjective-CA+noise}
    A deterministic CA on $\Sigma^{\ZZ^d}$ is specified by a local rule $f\colon\Sigma^N\to\Sigma$, where $N\subseteq\ZZ^d$ is finite.
    From a configuration $x$, the system evolves in one step to a new configuration $T(x)$, where
    \begin{equation}
        T(x)_k \isdef f\big((\sigma^k x)_N\big)\;,    \qquad\text{for $k\in\ZZ^d$.}
    \end{equation}
    A \emph{perturbation} of $T$ with (zero-range) \emph{noise} is a PCA $\Phi$ with local rule $\varphi(w,b)\isdef \theta\big(f(w),b\big)$, where $\theta\colon\Sigma\times\Sigma\to[0,1]$ is a stochastic matrix.  In words, each site is first updated deterministically via~$f$, and then independently resampled according to~$\theta$.  The global transition kernel of the PCA is given by
    \begin{equation}
        \Phi(x,E) \isdef \Theta\big(T(x),E\big) \;,
    \end{equation}
    for a configuration $x$ and a measurable set~$E$,
    where $\Theta$ denotes the product kernel defined by~$\theta$ (see \ac{eg},~\cite{MST2019}).

    If $T$ is surjective, it preserves the uniform Bernoulli measure~$\lambda_u$ on~$\Sigma^{\ZZ^d}$ (see \ac{eg},~\cite{Kari2026}).  Thus, if the noise kernel $\Theta$ preserves $\lambda_u$, so does the perturbed CA $\Phi$.  This is the case precisely when $\theta$ is doubly stochastic.  Marcovici, Sablik, and Taati used an entropy argument to show that, when $\theta$ is moreover strictly positive, the distribution of the system starting from any shift-invariant measure converges to~$\lambda_u$.  The shift-invariance assumption was later removed to establish full exponential ergodicity~\cite{Taati2021}.  The proofs of Theorems~\ref{thm:main:PCA:ergodicity} and~\ref{thm:main:IPS:ergodicity} are inspired by the latter proof.

    Note that the PCA in Example~\ref{exp:xor-with-noise} is a perturbation of a surjective CA with a doubly-stochastic noise.  The local rule of the CA and the noise matrix are the following:
    \begin{equation}
        f(w_0w_1)\isdef (w_0+w_1)\bmod{2} \;, \qquad
            \theta\isdef
                \begin{bmatrix}
                    1-\varepsilon & \varepsilon \\ \varepsilon & 1-\varepsilon
                \end{bmatrix} \;.
    \end{equation}

    More generally, the Bernoulli invariant measures of a given surjective CA have an algorithmic characterization based on conservation laws~\cite{KT2011,KT2015}.
    A Bernoulli measure $\lambda_q$ is stationary for the perturbation~$\Phi$ if $\lambda_q$ is invariant under~$T$ and $q\theta=q$.
    Theorem~\ref{thm:main:PCA:ergodicity} establishes the exponential ergodicity of all such perturbations as long as~$\theta$ is strictly positive.
\end{example}

Determining whether a probability measure $\mu$ is stationary for a PCA $\Phi$ requires verifying infinitely many identities $(\mu\Phi)([w])=\mu([w])$, one for each cylinder~$[w]$.  When $\mu$ is a Bernoulli measure, one may hope for a simpler condition involving only finitely many identities.  For the one-dimensional case, such a characterization was found by Piatetski-Shapiro~\cite[Chapter~16]{TVS+1990}.  He showed that, in order for a full-support Bernoulli or one-step Markov measure to be stationary for a one-dimensional PCA with dependence neighbourhood $N=\{0,1\}$ and alphabet~$\Sigma$, it is sufficient that the identity $(\mu\Phi)([w])=\mu([w])$ holds for all words $w\in\Sigma^*$ of length $\abs{w}\leq\abs{\Sigma}+1$.  The case of one-dimensional PCA with larger neighbourhoods or higher range Markov measures can be reduced to the latter case by dividing the lattice into blocks.

In two and higher dimensions, in contrast, there is  no hope of having a similar finitary characterization,
%as the following theorem illustrates.
as the following two algorithmic problems turn out to be both undecidable: %in any dimension $d\geq 2$:
\begin{enumerate}[label={P\arabic*. }]
    \item Is the Bernoulli measure with a given rational marginal~$q$ stationary for the PCA with a given positive, rational local transition rule~$\varphi$?
    \item Does the PCA with a given positive, rational local transition rule~$\varphi$ admit a stationary Bernoulli measure?
\end{enumerate}
In fact:
\begin{theorem}[Algorithmic indistinguishability of PCA with stationary Bernoulli measures]
\label{thm:PCA:undecidability}
    Let $d\geq 2$.  The following two classes of $d$-dimensional PCA with positive, rational local transition rules are computably inseparable:
    \begin{enumerate}[label={$\Alph*$.}]
        \item[$A$.] Those that admit a stationary Bernoulli measure with uniform marginal.
        \item[$B$.] Those that admit no stationary Bernoulli measures.
    \end{enumerate}
    In other words, there is no algorithm that, given a PCA $\Phi$ with a positive, rational local transition rule, outputs \texttt{\textup{YES}} if $\Phi$ belongs to class~$A$ and \texttt{\textup{NO}} if $\Phi$ belongs to class~$B$.
    If $\Phi$ does not belong to either $A$ or $B$, the algorithm is allowed to output arbitrarily or never halt.
\end{theorem}

\begin{proof}
    It is enough to prove the claim for~$d=2$.
    The question of whether a two-dimensional deterministic CA is surjective is known to be algorithmically undecidable~\cite{Kari1994}.  We show that this problem can be reduced to the problem of distinguishing between classes $A$ and $B$ of two-dimensional PCA.

    Let $F$ be a two-dimensional CA with local rule $f\colon\Sigma^N\to\Sigma$, and suppose we wish to determine if $F$ is surjective.  For $F$ to be surjective, $f$ must also be surjective.  Since the surjectivity of $f$ is algorithmically trivial, we can assume, without loss of generality, that $f$ is surjective.

    Let $\Phi$ be the perturbation of $F$ with memoryless noise with error probability $\kappa$ and uniform replacement distribution, where $\kappa$ is an arbitrary rational in $(0,1)$, say, $\kappa\isdef\sfrac{1}{2}$.
    %satisfying $0<\kappa<1/(\abs{\Sigma}+1)$, say, $\kappa\isdef 1/(\abs{\Sigma}+2)$.
    Thus, $\Phi$ is a PCA with local rule $\varphi\colon\Sigma^N\times\Sigma\to[0,1]$ given by $\varphi(w,b)\isdef\theta(f(w),b)$, where
    \begin{equation}
        \theta(a,b) \isdef
        \begin{cases}
            1-\kappa+\sfrac{\kappa}{\abs{\Sigma}}   & \text{if $b=a$,} \\
            \sfrac{\kappa}{\abs{\Sigma}}             & \text{otherwise,}
        \end{cases}
    \end{equation}
    is the transition matrix of the noise.  Note that $\varphi$ has positive, rational values.
    Let $\Theta$ denote the global kernel associated with~$\theta$.

    As discussed in Example~\ref{exp:surjective-CA+noise}, if $F$ is surjective, then the uniform Bernoulli measure is stationary for~$\Phi$.
    Conversely, let us show that, if $\Phi$ admits a stationary Bernoulli measure, $F$ must be surjective.

    The argument is a refinement of the one for Proposition~\ref{prop:PCA:decomposition:invariance}.
    Let $\lambda_q$ be a Bernoulli measure with marginal~$q$, and suppose that $\lambda_q\Phi=\lambda_q$.
    Let $\mu$ denote the image of $\lambda_q$ under $F$.  Then, $\mu\Theta=\lambda_q$.
    Let $p\colon\Sigma\to[0,1]$ denote the single-site marginal of~$\mu$.  Then, $p\theta=q$, hence $\lambda_p\Theta=\lambda_q$.  By Lemma~\ref{lem:noise-kernel:injectivity}, $\Theta$ acts injectively on probability measures, thus $\mu=\lambda_p$.
    Since $f$ is assumed to be surjective, $p(a)>0$ for every $a\in\Sigma$.
    Therefore, $\lambda_p$ has full support, which in turn implies that $F$ is surjective. 
\end{proof}

\subsection{IPS with stationary Bernoulli measures}
\label{sec:characterizations:IPS}

As noted in Observation~\ref{obs:IPS:local-preservation}, in order for a probability measure $\lambda$ to be stationary for an IPS $\Phi$, it is enough that $\lambda\widehat{\Phi}_k=\lambda$ for every $k\in\ZZ^d$.  Let us say that $\lambda$ is \emph{locally stationary} for $\Phi$ if this is the case.

% When $\Phi$ is defined by a local transition rule $\varphi\colon\Sigma^N\times\Sigma\to[0,1]$ and $\lambda=\lambda_p$ is a Bernoulli measure, local stationarity is equivalent to the following property:
% If $X_i$ ($i\in N$) are \ac{iid} random variables with distribution~$p$ and $(X_i)_{i\in N}\markovto[\varphi]Y_0$, then $Y_0$ also has distribution~$p$ and is independent of~$(X_i)_{i\in N\setminus\{0\}}$.

\begin{example}[Asynchronous XOR + noise]
    Consider the one-dimensional IPS with transition rule~$\varphi_\varepsilon$ introduced in Example~\ref{exp:xor-with-noise} and clock rate~$1$.
    The uniform Bernoulli measure $\lambda_{\sfrac{1}{2}}$ is locally stationary.  Indeed, if
    \begin{equation}
        \begin{array}{C{3em}C{3em}C{3em}C{3em}C{3em}C{3em}C{3em}}
            \cdots & X_{k-2} & X_{k-1} & X_k & X_{k+1} & X_{k+2} & \cdots
        \end{array}
    \end{equation}
    are \ac{iid} uniform Bernoulli random variables, so are
    \begin{equation}
        \begin{array}{C{3em}C{3em}C{3em}C{3em}C{3em}C{3em}C{3em}}
            \cdots & X_{k-2} & X_{k-1} & Y_k & X_{k+1} & X_{k+2} & \cdots
        \end{array}
    \end{equation}
    where $Y_k$ is sampled from $\varphi_\varepsilon(X_kX_{k+1},\cdot)$.
    %where $Y_0\isdef (X_0+X_0) \bmod{2}$, and the noise also preserves the uniform Bernoulli distribution.
\end{example}

One may suspect that for a Bernoulli (or Markov) measure, stationarity and local stationarity are equivalent.  The following example refutes this conjecture.

\begin{example}[Stationary but not locally stationary]
\label{exp:IPS:stationary-but-not-locally}
    Let $0<\varepsilon\leq\sfrac{1}{2}$.
    Let $\Phi_\varepsilon$ be the one-dimensional IPS with alphabet $\Sigma\isdef\{\symb{0},\symb{1}\}$, dependence neighbourhood $N\isdef\{-1,0,1\}$, and local rule %$\varphi\colon\Sigma^N\times\Sigma\to(0,1)$
    \begin{equation}
        \varphi_\varepsilon(w_{-1}w_0w_1,\cdot) \isdef
            \varepsilon\delta_{w_1} + \varepsilon\delta_{\overline{w}_{-1}} + (1-2\varepsilon)\delta_{w_0}
    \end{equation}
    and clock rate~$1$,
    where, for $a\in\Sigma$, $\delta_a$ denotes the point mass at $a$, and $\overline{a}$ denotes the opposite of $a$ (\ac{ie}, $\overline{\symb{1}}=\symb{0}$ and $\overline{\symb{0}}=\symb{1}$).
    In words, the system evolves as follows.
    %We have independent Poisson clocks with rate~$1$ at every site.
    Every time the clock at site $k$ ticks, we update the symbol at site~$k$ in the following manner:
    \begin{itemize}
        \item With probability $\varepsilon$, we copy the value at site $k+1$ into site $k$.
        \item With probability $\varepsilon$, we copy the flipped version of the value at site $k-1$ into site~$k$.
        \item With probability $1-2\varepsilon$, we leave the value at site $k$ unchanged.
    \end{itemize}
    We claim that, for this IPS, the uniform Bernoulli measure $\lambda_{\sfrac{1}{2}}$ is stationary but not locally stationary.
    
    That $\lambda_{\sfrac{1}{2}}$ is not locally stationary is clear, because updating the value at site~$k$ introduces dependencies between site $k$ and its two neighbours.
    To see that $\lambda_{\sfrac{1}{2}}$ is stationary, note that, by the virtue of the colouring and superposition properties of the Poisson processes (see \ac{eg},~\cite{Kingman1993}), this model has the following equivalent description.
    For each $k$, there is a Poisson clock $\xi'_{k,k+1}$ with rate~$2\varepsilon$ attached to the pair $(k,k+1)$, and all the clocks are independent.  Whenever the clock at $(k,k+1)$ ticks, we flip a fair coin and
    \begin{itemize}
        \item With probability $\sfrac{1}{2}$, we copy the value at site $k+1$ into site~$k$.
        \item With probability $\sfrac{1}{2}$, we copy the flipped version of the value at site $k$ into site~$k+1$.
    \end{itemize}
    It is easy to verify that, if the values at $k$ and $k+1$ are independent uniform Bernoulli, then after the latter updating step, their joint distribution remains the same.
\end{example}

% \comment{To add:
% \begin{itemize}
%     \item Conjecture that IPS with stationary Bernoulli measure still have a finitary characterization.
%     \item Result of Fredes and Marckert~\cite{FM2020}.
% \end{itemize}
% }

\begin{question}
    Do the (positive-rate) IPS that admit stationary Bernoulli measures have a finitary characterization in every dimension?
\end{question}

In one dimension, this question was resolved by Fredes and Marckert~\cite{FM2020}, who provided finitary conditions for the stationarity of Bernoulli and full-support Markov measures for a more general class of finite-range IPS in which multiple sites are allowed to be updated simultaneously.

% \comment{To be written!}

% Theorem~\ref{thm:PCA:undecidability} shows that the positive-rate PCA that admit stationary Bernoulli measures form a rich family of models.
% In contrast, the positive-rate IPS that admit stationary Bernoulli measures seem to have a simple characterization.

% Let $\Phi$ be an IPS with a local transition rule $\varphi\colon\Sigma^N\times\Sigma\to[0,1]$ and let $\widehat{\Phi}_k$ be the kernel for updating the symbol at position~$k$ using~$\varphi$.
% We say that $\Phi$ \emph{locally preserves} a probability measure $\lambda$ if $\lambda\widehat{\Phi}_k=\lambda$ for every $k\in\ZZ^d$.  Clearly, this is equivalent to the single condition $\lambda\widehat{\Phi}_0=\lambda$ as long as $\lambda$ is shift-invariant.  By Observation~\ref{obs:IPS:local-preservation}, local preservation implies preservation.

% \begin{question}	
% 	Characterize the positive-rate IPS that admit stationary Bernoulli measures.
% 	Is it true that a positive-rate IPS preserves a stationary Bernoulli measure if and only if it does so locally?
% \end{question}

% In the case of binary alphabet, a necessary and sufficient condition is given by Stavskaya (see Toom et al.~\cite[Chapter~17]{TVS+1990}).

% References ***************************************
\addcontentsline{toc}{section}{References}
\bibliographystyle{plainurl}
\bibliography{bibliography}
% --------------------------------------------------

\bigskip\bigskip
\Addresses

\end{document}